\documentclass[a4paper, final]{amsart}

\usepackage{amssymb}
\usepackage{amsmath}
\usepackage{amsfonts}
\usepackage{dsfont}
\usepackage{amsthm}

\DeclareMathOperator*{\var}{var}
\DeclareMathOperator*{\sgn}{sgn}

\newtheorem{thm}{Theorem}[section]
\newtheorem{prop}[thm]{Proposition}
\theoremstyle{definition}
\newtheorem{dfn}{Definition}[section]
\theoremstyle{remark}

\theoremstyle{plain}
\newtheorem{lem}[thm]{Lemma}

\begin{document}

\title[Feynman-Kac for PAM driven by fractional noise]{Feynman-Kac representation for the parabolic Anderson model driven by fractional noise}

\author[K. Kalbasi]{Kamran Kalbasi}

\author[T. Mountford]{Thomas S. Mountford}

\address{Department of Mathematics, Ecole Polytechnique F\'{e}d\'{e}rale de Lausanne, CH-1015 Lausanne, Switzerland}


\begin{abstract}
We consider the parabolic Anderson model driven by fractional noise:
$$
\frac{\partial}{\partial t}u(t,x)= \kappa \boldsymbol{\Delta} u(t,x)+ u(t,x)\frac{\partial}{\partial t}W(t,x)
\qquad x\in\mathds{Z}^d\;,\; t\geq 0\,,
$$
where $\kappa>0$ is a diffusion constant, $\boldsymbol{\Delta}$ is the discrete Laplacian defined by $\boldsymbol{\Delta} f(x)= \frac{1}{2d}\sum_{|y-x|=1}\bigl(f(y)-f(x)\bigr)$, and $\{W(t,x)\;;\;t\geq0\}_{x \in \mathds{Z}^d}$ is a family of independent fractional Brownian motions with Hurst parameter $H\in(0,1)$, indexed by $\mathds{Z}^d$. We make sense of this equation via a Stratonovich integration obtained by approximating the fractional Brownian motions with a family of Gaussian processes possessing absolutely continuous sample paths. We prove that the Feynman-Kac representation
\begin{equation}
u(t,x)=\mathbb{E}^x\Bigl[u_o(X(t))\exp \int_0^t W\bigl(\mathrm{d}s, X(t-s)\bigr)\Bigr]\,,
\end{equation}
is a mild solution to this problem. Here $u_o(y)$ is the initial value at site $y\in\mathds{Z}^d$, $\{X(t)\;;\;t\geq0\}$ is a simple random walk with jump rate $\kappa$, started at $x \in \mathds{Z}^d$ and independent of the family $\{W(t,x)\;;\;t\geq0\}_{x\in\mathds{Z}^d}$ and $\mathbb{E}^x$ is expectation with respect to this random walk. We give a unified argument that works for any Hurst parameter $H\in (0,1)$.

\end{abstract}

\maketitle

\section{Introduction}
The parabolic Anderson model(PAM)named after the Nobel laureate physicist Philip W. Anderson, is the parabolic partial differential equation
\begin{equation}\label{general PAM}
\frac{\partial}{\partial t}u(t,x)= \kappa \boldsymbol{\Delta} u(t,x)+ \xi(t,x)\,u(t,x),
\qquad x\in\mathds{Z}^d\;,\; t\geq 0\,,
\end{equation}
where $\kappa>0$ is a diffusion constant and $\boldsymbol{\Delta}$ is the discrete Laplacian defined by $\boldsymbol{\Delta} f(x)= \frac{1}{2d}\sum_{|y-x|=1}\bigl(f(y)-f(x)\bigr)$. The potential $\{\xi(t,x)\}_{t,x}$ can be a random or deterministic field and even a Schwartz distribution.

The parabolic Anderson model which has been extensively studied, particularly in the last twenty years, has many applications and connections to problems in chemical kinetics, magnetic fields with random flow and the spectrum of random Schr\"odinger operators, to mention a few. The solution $u(t,x)$ of \eqref{general PAM} has also a population dynamics interpretation as the average number of particles at site $x$ and time $t$ conditioned on a realization of the medium $\xi$, where the particles perform branching random walks in random media. In this case, the first right-hand-side term of \eqref{general PAM} signifies the diffusion and the second term represents the birth/death of the particles. We refer to the classical work of Carmona and Molchanov \cite{Carmona} and the survey by G\"artner and K\"onig \cite{GaertnerKoenig}.

We consider the parabolic Anderson model with the potential $\xi(t,x):=\frac{\partial}{\partial t}W(t,x)$ for $x\in\mathds{Z}^d$ and $t\geq0$, where $\{W(t,x)\;;\;t\geq0\}_{x \in \mathds{Z}^d}$ is a family of independent fractional Brownian motions(fBM) of Hurst parameter $H$, indexed by $\mathds{Z}^d$.

As the paths of fBM are almost surely nowhere differentiable, this equation doesn't make sense a priori in the classical sense and we reformulate it in a mild sense:
\begin{equation}\label{e1}
    \left\{
        \begin{aligned}
        & u(t,x)-u(0,x)=\int_0^t \boldsymbol{\Delta} u(s,x)\,\mathrm{d}s+\int_0^t u(s,x)\, W(\mathrm{d}s,x)\\
        &u(0,x)=u_o(x)
        \end{aligned}
    \right.\,,
\end{equation}
where the stochastic integral is Stratonovich type in the sense that the fractional Brownian motion is approximated by a family of smooth processes $\{W^{\varepsilon}\}_{\varepsilon>0}$ and the integral $\int u\, \mathrm{d}W$ is defined by the limit of the family $\{\int u\, \mathrm{d}W^{\varepsilon}\}_{\varepsilon>0}$ as $\varepsilon$ tends to zero. We assume that $u_o(\cdot)$ is a bounded measurable function.

We will show that the following Feynman-Kac formula gives a solution to \eqref{e1}:
\begin{equation}\label{Feynman-Kac}
u(t,x)=\mathbb{E}^x\Bigl[u_o(X(t))\exp \int_0^t W\bigl(\mathrm{d}s, X(t-s)\bigr)\Bigr]\,,
\end{equation}
where $X(t)$ is a simple random walk with jump rate $\kappa$, started at $x \in \mathds{Z}^d$ and independent of the family $\{W(t,x)\;;\;t\geq0\}_{x\in\mathds{Z}^d}$ and $\mathbb{E}^x$ is expectation with respect to this random walk. Here the stochastic integral is nothing other than a summation. Indeed, suppose that $\{t_i\}_{i=1}^{n}$ are the jump times of the time-reversed random walk $\{X(t-s)\;,\; s\in[0,t]\}$ with the additional convention $t_0:=0$ and $t_{n+1}:=t$. Let also $x_i$ for $i=0,\cdots,n$ be the value of $\{X(t-\cdot)\}$ at time interval $[t_i,t_{i+1})$. Then we have
$$
\int_0^t W\bigl(\mathrm{d}s, X(t-s)\bigr)=\sum_{i=0}^{n}\bigl(W(t_{i+1},x_i)-W(t_{i},x_i)\bigr)\,.
$$

Carmona and Molchanov in their classical memoir \cite{Carmona} proved that for bounded $u_o$ and $H=1/2$ i.e.
standard Brownian motion, the Feynman-Kac formula \eqref{Feynman-Kac} solves equation \eqref{e1}. The asymptotic behavior of the Feynman-Kac expression \eqref{Feynman-Kac} as the partition function of a directed polymer in a random environment has been studied in \cite{Viens}, but its connection with the PAM has not been investigated.
The Feynman-Kac representation for PAM on $\mathds{R}^d$ driven by fractional noise was
established in \cite{Nualart1} for Hurst parameters $H\geq 1/2$ and in \cite{Nualart2} for $H\geq 1/4$.
Our method is able to prove this property without any restriction on $H$ due to the fact that in the discrete case one deals with locally constant random walk instead of Brownian motion which is only locally $\alpha$-H\"older continuous for $\alpha<1/2$.

The paper is organized as follows:\\
In section \ref{Preliminaries} we collect some important background material that we will use in the succeeding sections.\\
In section \ref{Setting} we outline our methodology including the approximation scheme that we apply to fractional Brownian motion. We show that the problem reduces to demonstrating the convergence of three expressions $u_\varepsilon$, $V_{1,\varepsilon}$ and $V_{2,\varepsilon}$.\\
It section \ref{Approximation rate}, we prove that piecewise-constant integrals with respect to the approximating processes introduced in section \ref{Setting} converge to integrals with respect to fractional Brownian motion.\\
The remaining chapters are devoted to showing the convergence of $u_\varepsilon$, $V_{1,\varepsilon}$ and $V_{2,\varepsilon}$.

\section{Preliminaries}\label{Preliminaries}

A Gaussian random process $W(\cdot)$ is called a fractional Brownian motion of Hurst parameter $H \in (0,1)$, if it has continuous sample paths and its covariance function is of the following form:
\begin{equation*}
\mathbb{E}[W(t)W(s)]=R_H(t,s):=\frac{1}{2}(|t|^{2H}+|s|^{2H}-|t-s|^{2H}).
\end{equation*}
This process was first introduced by Kolmogorov in \cite{Kolmogorov}, but the term ``Fractional Brownian motion'' was coined by Mandelbrot and Van Ness in \cite{Mandelbrot}.

Let $\{W(t,x); t\in\mathds{R}\}_{x\in \mathds{Z}^d}$ be a family of independent fractional Brownian motions indexed by $x\in \mathds{Z}^d$ all with Hurst parameter $H$.\\
Similar to \cite{Nualart2}, let $\mathcal{H}$ be the Hilbert space defined by the completion of the linear span of indicator functions $\mathbf{1}_{[0,t]\times\{x\}}$ for $t \in \mathds{R}$ and $x\in \mathds{Z}^d$ under the scalar product
\begin{equation*}
\langle\mathbf{1}_{[0,t]\times\{x\}} , \mathbf{1}_{[0,s]\times\{y\}}\rangle_\mathcal{H}=R_H(t,s)\,\delta_x(y)\,,
\end{equation*}
where $\delta$ is the Kronecker delta. Here we assume the convention $\mathbf{1}_{[0,t]\times\{x\}} := -\mathbf{1}_{[t,0]\times\{x\}}$ for negative $t$. The mapping $\mathbf{W}(\mathbf{1}_{[0,t]\times\{x\}}):= W(t,x)$ can be extended to a linear isometry from $\mathcal{H}$ onto the Gaussian Hilbert space spanned by $\{W(t,x)\;;\;t\in\mathds{R}\,,\,x\in \mathds{Z}^d\}$.

Similar to \cite{Nualart2}, for any piecewise constant function $X:\mathds{R}\rightarrow \mathds{Z}^d$, and every $s\in\mathds{R}$, $x\in\mathds{Z}^d$ and $\varepsilon>0$ we define the following functions on $\mathds{R}\times\mathds{Z}^d$:
\begin{equation}\label{gEpsilon}
g_{s,x}^\varepsilon (r,z):= \frac{1}{2\varepsilon} \mathbf{1}_{[s-\varepsilon,s+\varepsilon]}(r)\,\delta_x(z)\,,
\end{equation}
\begin{equation}\label{gX}
g_{s,x}^X (r,z):= \mathbf{1}_{[0,s]}(r)\, \delta_{X(s-r)}(z)\,,
\end{equation}
\begin{equation}\label{gEpsilonX}
g_{s,x}^{\varepsilon, X} (r,z):= \int_0^s \frac{1}{2\varepsilon} \mathbf{1}_{[\theta-\varepsilon,\theta+\varepsilon]}(r)\, \delta_{X(s-\theta)}(z)\, \mathrm{d}\theta\,.
\end{equation}
It can be easily shown that $g_{s,x}^\varepsilon$, $g_{s,x}^X$ and $g_{s,x}^{\varepsilon, X}$ are all in $\mathcal{H}$, and moreover
\begin{equation*}
\mathbf{W}(g_{s,x}^\varepsilon)= \dot{W}_\varepsilon(s,x)\,,
\end{equation*}
\begin{equation*}
\mathbf{W}(g_{s,x}^X)= \int_0^s W\bigl(\mathrm{d}\theta,X(s-\theta)\bigr)\,,
\end{equation*}
and
\begin{equation*}
\mathbf{W}(g_{s,x}^{\varepsilon, X}) = \int_0^s \dot{W}_{\varepsilon}\bigl(\theta,X(s-\theta)\bigr)\mathrm{d}\theta\,,
\end{equation*}
where $\dot{W_{\varepsilon}}(t,x):=\frac{1}{2\varepsilon} \bigl(W(t+\varepsilon,x)-W(t-\varepsilon,x)\bigr)$ for any $t\in\mathds{R}$ and $x\in\mathds{Z}^d$.

%

Let $\mathbf{G}$ be a Gaussian Hilbert space, $\mathbf{H}$ a Hilbert space and $\mathbf{W}:\mathbf{H}\rightarrow \mathbf{G}$ a Hilbert space isometry between $\mathbf{H}$ and $\mathbf{G}$. By a Gaussian Hilbert space we mean a set of zero-mean Gaussian random variables which is a Hilbert space with respect to covariance as its inner product \cite{Janson}. Define $\mathcal{S}$ as the space of random variables $F$ of the form:
\begin{equation*}
F=f\bigl(\mathbf{W}(\varphi_1) , \dots , \mathbf{W}(\varphi_n)\bigr)\,,
\end{equation*}
where $\varphi_i \in \mathbf{H}$ and $f\in C^\infty(\mathds{R}^n)$ with $f$ and all its partial derivatives having polynomial growth. The Malliavin derivative of $F$ denoted by $\nabla F$, is defined (see e.g. \cite{Nualart2,Janson,Nualart3,Sanz}) as the $\mathbf{H}$-valued random variable given by
\begin{equation*}
\nabla F:= \sum _ {i=1}^{n} \frac{\partial f}{\partial x_i}\bigl(\mathbf{W}(\varphi_1)  , ... , \mathbf{W}(\varphi_n) \bigr)\, \varphi_i\,.
\end{equation*}
The operator $\nabla$ extends to the Sobolev space $\mathbb{D}^{1,2}$ which is defined as the closure of $\mathcal{S}$ with respect to the following norm \cite{Nualart2,Janson}:
\begin{equation*}
\|F\|_{1,2}= \sqrt{\mathbb{E}(F^2)+\mathbb{E}(\|\nabla F\|^2_\mathbf{H})}\,.
\end{equation*}

The divergence operator $\boldsymbol{\delta}$ is the adjoint of the derivative operator $\nabla$, determined by the duality relationship \cite{Nualart2,Janson}
\begin{equation*}
\mathbb{E}(\boldsymbol{\delta}(u) F) = \mathbb{E}(\langle \nabla F ,  u\rangle_\mathbf{H}) \quad\text{for every}\;F\in \mathbb{D}^{1,2}.
\end{equation*}
The space of $\mathbf{H}$-valued Malliavin derivable $\mathcal{L}^2$ random variables with $\mathcal{L}^2$ derivatives, denoted by $\mathbb{D}^{1,2}(\mathbf{H})$, is contained in the domain of $\boldsymbol{\delta}$, and moreover for any $u\in \mathbb{D}^{1,2}(\mathbf{H})$, we have
\begin{equation}\label{divergence contractivity}
\mathbb{E}\bigl(\boldsymbol{\delta}(u)^2\bigr)\leq
\mathbb{E}\bigl(\|u\|^2_\mathbf{H}\bigr)+\mathbb{E}\bigl(\|\nabla u\|^2_{\mathbf{H}\otimes\mathbf{H}}\bigr)\,.
\end{equation}
For any random variable $F \in \mathbb{D}^{1,2}$ and $\varphi \in \mathbf{H}$ the change of variable formula \cite{Nualart2,Janson}:
\begin{equation}\label{change of variable}
F \mathbf{W}(\varphi) = \boldsymbol{\delta}(F \varphi ) + \langle \nabla F, \varphi\rangle_\mathbf{H}\,.
\end{equation}\\
For more on Malliavin calculus we refer to \cite{Janson,Nualart3}.

We will use the following lemma in several occasions:
\begin{lem}\label{Commutativity of integration with operators}
Let $(M,\mathcal{M},\mu)$ be a measure space and $B$, $B'$ be Banach spaces. Let also $\boldsymbol{\Lambda}:B \rightarrow B'$ be a continuous linear operator and $f:M\rightarrow B$ a separably-valued measurable function, i.e. there exists a separable subspace $B_1$ of $B$ such that $f \in B_1$ almost surely. If $\int \|f\|_B \mathrm{d}\mu < \infty$ then
$$
\boldsymbol{\Lambda} \int f \mathrm{d}\mu = \int \boldsymbol{\Lambda} f \mathrm{d}\mu\,.
$$
\end{lem}
\begin{proof}
As $f$ is separably-valued, there exists \cite{Janson,LedouxTalagrand} a sequence of simple functions $\{u_n\}_n$ of the form $\sum_i \mathbf{1}_{A_i} h_i$ with $A_i \in \mathcal{M}$ and $h_i \in B$ with the property that
$$
\int \| u_n - f \|_B \mathrm{d}\mu \longrightarrow 0  \qquad as \qquad n \rightarrow \infty
\,.$$
As $\boldsymbol{\Lambda}$ is linear, it commutes with integration on $\{u_n\}_n$. As $\boldsymbol{\Lambda}$ is continuous we have $\|\boldsymbol{\Lambda} (x)\|_B \leq C \|x\|_{B'}$ for some positive constant $C$, so
$$
\int \| \boldsymbol{\Lambda}(u_n - f) \|_{B'} \mathrm{d}\mu \leq C \int \| (u_n - f) \|_B \mathrm{d}\mu
$$
and also
\begin{equation*}
\begin{split}
&\| \boldsymbol{\Lambda} \int (u_n - f)  \mathrm{d}\mu \|_{B'} \leq C \| \int (u_n - f)  \mathrm{d}\mu \|_B\\
& \qquad \qquad \leq  C  \int \|u_n - f\|_B  \mathrm{d}\mu\,.
\end{split}
\end{equation*}
Hence $\boldsymbol{\Lambda}$ commutes with integration for $f$ too.
\end{proof}

\section{Setting}\label{Setting}
As explained in the previous section we aim to approximate the fractional Brownian motions with a family of smooth Gaussian processes. There are two obvious ways to approximate a (fractional) Brownian motion. First the so-called Wong-Zakai approximation scheme \cite{Twardowska} which is the piecewise linear approximation of (fractional) Brownian motion paths. The second natural scheme is as follows: The time derivative of a fractional Brownian motion does not exist in the classical sense but only in the distributional sense. The idea is to approximate the `derivative' of the fractional Brownian motion and then integrate it. Indeed we define the approximate derivative of $W(\cdot,x)$ as $\dot{W_{\varepsilon}}(\cdot,x)$
\begin{equation}\label{Derivative approximation formula}
\dot{W_{\varepsilon}}(t,x):=\frac{1}{2\varepsilon} \bigl(W(t+\varepsilon,x)-W(t-\varepsilon,x)\bigr)\,.
\end{equation}
Proposition \ref{main variance inequality} shows in particular that the integral of this family of Gaussian processes converges to fractional Brownian motion.

While the first scheme doesn't seem to be easy to work with, the second one has been proved to be very suitable in our setting where we use the Wiener space technics and Malliavin calculus \cite{Nualart2}.

Now let first replace the fBM family $\{W(\cdot,x)\}_{x \in \mathds{Z}^d}$ in equation \eqref{e1} by a family of absolutely continuous functions $\{\Xi(\cdot,x)\}_{x \in \mathds{Z}^d}$, or equivalently replace the family of fractional noises $\{\frac{\partial}{\partial t}W(\cdot,x)\}_{x \in \mathds{Z}^d}$ by a family of locally integrable functions $\{\xi(\cdot,x)\}_{x \in \mathds{Z}^d}$ where $\Xi(t,x)=\int_0^t\xi(s,x)\mathrm{d}s$ for every $x$ and $t$. Carmona and Molchanov in \cite{Carmona} showed that the Feynman-Kac formula
$$
\mathcal{F}(\Xi):=\mathbb{E}^x\Bigl[u_o(X(t))\exp \int_0^t \Xi\bigl(\mathrm{d}s ,X(t-s)\bigr)\Bigr]=\mathbb{E}^x\Bigl[u_o(X(t))\exp \int_0^t \xi(s,X(t-s)\mathrm{d}s\Bigr]
$$
solves the PAM driven by the potential $\{\xi(\cdot,x)\}_{x \in \mathds{Z}^d}$ if this expression is finite for every $x$ and $t$.

If we approximate every fractional Brownian motion $W(\cdot,x)$ by a family of stochastic processes $\{W^{\varepsilon}(\cdot,x)\}_{\varepsilon>0}$ which converge to $W(\cdot,x)$ and with the property that every $W^{\varepsilon}(\cdot,x)$ has absolutely continuous sample paths, we expect that $\mathcal{F}(W^{\varepsilon})$ should also converge $\mathcal{F}(W)$. On the other hand, if we denote by $u^{\varepsilon}$ the solution of equation \eqref{e1} with $W$ replaced by $W^{\varepsilon}$,  we also expect that $u^{\varepsilon}$ should converge to the solution of \eqref{e1} with the integral understood in the Stratonovich sense. The reason is that for the stochastic differential equations with Brownian motion or more generally semi-martingale terms, if the Brownian motions (semi-martingales) are approximated by a family of processes with absolutely continuous sample paths, the sequence of solutions converges to the Stratonovich solution of the original differential equation \cite{Stroock,Protter}. Note that for each sample path of an such processes, a solution in the classical sense exists.

So we consider the approximation scheme of equation \eqref{Derivative approximation formula}. In the rest of the paper, without any loss of generality we will assume that $\kappa=1$. We also denote by $\mathbb{E}$ the expectation with respect to the fractional Brownian field and by $\mathbb{E}^x$ the expectation with respect to the random walk $X(\cdot)$.

Let
\begin{equation}\label{FeynmanKac for epsilon}
u_\varepsilon(t,x):=\mathbb{E}^x\Bigl[u_o(X(t))\exp \int_0^t \dot{W_{\varepsilon}}\bigl(s,X(t-s)\bigr)\mathrm{d}s\Bigr]\,,
\end{equation}
where $\dot{W_{\varepsilon}}$ is defined in \eqref{Derivative approximation formula}.

By lemma \ref{uniform boundedness}, we have $\mathbb{E}|u_\varepsilon(t,x)|<\infty$ for every $x$ and $t$. So almost surely, $u_\varepsilon(t,x)$ is finite for every $x$ and $t$. On the other hand, the sample paths of $\dot{W_{\varepsilon}}$ are locally integrable. So by the above mentioned theorem of Carmona and Molchanov \cite{Carmona} the field $\{u_\varepsilon(t,x)\}_{x,t}$ solves the following equation
\begin{equation}\label{ApproxEqu}
    \left\{
        \begin{aligned}
        &\frac{\partial u_{\varepsilon}}{\partial t}=\boldsymbol{\Delta} u_{\varepsilon}+u_{\varepsilon} \dot{W_{\varepsilon}}\\
        &u_{\varepsilon}(0,x)=u_o(x)\,.
        \end{aligned}
    \right.
\end{equation}

We aim to show that \eqref{Feynman-Kac} gives a solution to \eqref{e1} with the Stratonovich integral $\int_0^t u(s,x)W(\mathrm{d}s, x)$ defined in the following natural manner which was also used in \cite{Nualart2}.
\begin{dfn}
For a random field $u=\{u(t,x)\;;\; t\in \mathds{R},x \in \mathds{Z}^d\}$, the Stratonovich integral
$$
\int_0^t u(s,x)W(\mathrm{d}s, x)
$$
is defined \cite{Nualart2} as the following $\mathcal{L}^2$ limit (if it exists)
$$
\lim_{\varepsilon \rightarrow 0}\int_0^t u(s,x)\dot{W_{\varepsilon}}\bigl(s, x\bigr)\mathrm{d}s\,.
$$
\end{dfn}
Using the same methodology of \cite{Nualart2} we will show that the Stratonovich integral of the Feynman-Kac formula \eqref{Feynman-Kac} exists and moreover it satisfies \eqref{e1}.

Indeed equation \eqref{ApproxEqu} can be integrated to
\begin{equation}\label{PA equation for epsilon}
u_\varepsilon(t,x)- u_o(x)=\int_0^t \boldsymbol{\Delta} u_\varepsilon (s,x)\mathrm{d}s + \int_0^t u_\varepsilon(s,x)\dot{W_{\varepsilon}}(s, x)\mathrm{d}s\,.
\end{equation}
Once we show that $u_{\varepsilon}$ (given by \eqref{FeynmanKac for epsilon}) converges to $u$ (given by \eqref{Feynman-Kac}) in  $\mathcal{L}^2$ sense and uniformly in $t \in [0,T]$ as  $\varepsilon$ goes down to zero, along with equation \eqref{PA equation for epsilon}, it would imply the $\mathcal{L}^2$-convergence of $\int \bigl(u_{\varepsilon}\dot{W_{\varepsilon}} \bigr)$ to some random variable. If moreover one shows that $\int \bigl(u_{\varepsilon}\dot{W_{\varepsilon}}-u\dot{W_{\varepsilon}} \bigr)$ converges in $\mathcal{L}^2$ to zero, it would imply the convergence of $\int \bigl(u\dot{W_{\varepsilon}} \bigr)$ and hence the existence of the Stratonovich integral $\int u \,\mathrm{d}W$. But this means that $u$ satisfies equation \eqref{e1}.

Let $g_{s,x}^\varepsilon$ be defined as in equation \eqref{gEpsilon}. So we have
$\mathbf{W}(g_{s,x}^\varepsilon)= \dot{W}_\varepsilon(s,x)$ and by the change of variable formula \eqref{change of variable} we obtain
$$
\begin{aligned}
&u_{\varepsilon}(s,x) \dot{W_{\varepsilon}}(s,x)-u(s,x)\dot{W_{\varepsilon}}(s,x) = \tilde{u}_\varepsilon(s,x) \mathbf{W}(g^{\varepsilon}_{s,x})\\
& \qquad = \boldsymbol{\delta}(\tilde{u}_\varepsilon(s,x) g^{\varepsilon}_{s,x} ) + \langle \nabla \tilde{u}_\varepsilon(s,x), g^{\varepsilon}_{s,x}\rangle_\mathcal{H}\,,
\end{aligned}
$$
where $\tilde{u}_\varepsilon:=u_{\varepsilon}-u$.

Hence it suffices to show that $V_{1,\varepsilon} := \int_0^t \boldsymbol{\delta}(\tilde{u}_\varepsilon(s,x) g^{\varepsilon}_{s,x} ) \mathrm{d}s$  and $V_{2,\varepsilon} := \int_0^t \langle \nabla \tilde{u}_\varepsilon(s,x), g^{\varepsilon}_{s,x}\rangle_\mathcal{H} \mathrm{d}s$ both converge to zero as $\varepsilon$ goes to zero. In sections \ref{convergence of u_epsilon}, \ref{convergence of v_1} and \ref{convergence of v_2} we will deal with the convergence of $u_\varepsilon$, $V_{1,\varepsilon}$ and $V_{2,\varepsilon}$.

\section{Approximation rate}\label{Approximation rate}
In this section we prove the following theorem that establishes the approximation of $W(\mathrm{d}s)$ by $\dot{W_{\varepsilon}}(s)\mathrm{d}s$. In the proof we will use some ideas of \cite{Nualart2} as well as simple properties of random walk.
\begin{prop}\label{main variance inequality}
Let $t$, $T$, $t_1$, $t_2$, ..., $t_N$ be some positive real numbers with $t_0:=0<t_1 < \dotsb < t_{N}<t_{N+1}:=t \leq T$ and $X(\cdot)$ a jump function on $[0,t]$ with values in $\mathds{Z}^d$ and jump times $\{t_1, ..., t_{N}\}$, i.e. $X(s)=x_i \in \mathds{Z}^d$ for $s \in (t_i, t_{i+1}]$. Then
\begin{equation*}
\mathbb{E}\Bigl|\int_0^t \dot{W_{\varepsilon}}\bigl(s,X(s)\bigr)\mathrm{d}s - \int _0^t W\bigl(\mathrm{d}s, X(s)\bigr)\Bigr|^2 \leq C N^2 \varepsilon^{\min\{2H,1\}}\,,
\end{equation*}
where $C$ is a constant depending only on $T$ and $H$ and
\begin{equation*}
\int _0^t W\bigl(\mathrm{d}s, X(s)\bigr)=\sum_{i=0}^N\bigl(W(t_{i+1},x_i)-W(t_i,x_i)\bigr)\,.
\end{equation*}
\end{prop}
\begin{proof}
First we show that for every $t_1$ and $t_2$, $t_1<t_2\leq T$, and any fractional Brownian motion $W(\cdot)$ with Hurst parameter $H \in (0,1)$ we have
\begin{equation}\label{piecewise}
\mathbb{E}\Bigl|W(t_2) - W(t_1) -\int_{t_1}^{t_2} \dot{W_{\varepsilon}}(\theta)\mathrm{d}\theta \Bigr|^2 \leq C \varepsilon^{\min\{2H,1\}}\,,
\end{equation}
where $\dot{W_{\varepsilon}}$ is the symmetric $\varepsilon$-derivative of $W$:
\begin{equation*}
\dot{W_{\varepsilon}}(t):=\frac{1}{2\varepsilon} \bigl(W(t+\varepsilon)-W(t-\varepsilon)\bigr)
\end{equation*}
and $C$ is some positive constant depending only on $T$ and $H$.
We have to calculate and bound
\begin{equation}\label{piecewise terms}
\begin{split}
& \mathbb{E}\Bigl|W(t_2) - W(t_1) -\int_{t_1}^{t_2} \dot{W_{\varepsilon}}(\theta)\mathrm{d}\theta \Bigr|^2 = \mathbb{E}\Bigl|W(t_2)-W(t_1)\Bigr|^2 \\
& \qquad + \int_{t_1}^{t_2} \int_{t_1}^{t_2} \mathbb{E}\Bigl[\dot{W_{\varepsilon}}(\theta)\dot{W_{\varepsilon}}(\eta)\Bigr]\mathrm{d}\theta \,\mathrm{d}\eta -2 \int_{t_1}^{t_2} \mathbb{E}\Bigl[\bigl(W({t_2})-W({t_1})\bigr)\dot{W_{\varepsilon}}(\theta)\Bigl]\mathrm{d}\theta\,.
\end{split}
\end{equation}
Let $\mathfrak{S}_1$ and $\mathfrak{S}_2$ be the first and second terms on the right hand side of this equation and $\mathfrak{S}_3$ be the third term without its $-2$ factor.\\
Using the following equality
\begin{equation*}
\mathbb{E}\Bigl[\bigl(W(a) - W(b)\bigr)\bigl(W(c) - W(d)\bigr) \Bigr] = \frac{1}{2} \Bigl[|a-d|^{2H}+|b-c|^{2H}-|a-c|^{2H}-|b-d|^{2H}\Bigr]
\end{equation*}
we have:
\begin{equation*}
\mathfrak{S}_1=|t_2-t_1|^{2H},
\end{equation*}
\begin{equation*}
\mathfrak{S}_2=\int_{t_1}^{t_2} \int_{t_1}^{t_2} \frac{1}{8\varepsilon^2}\Bigl[|s-\eta+2\varepsilon|^{2H}+|\eta-s+2\varepsilon|^{2H}-
2|s-\eta|^{2H}\Bigr]\mathrm{d}\eta \,\mathrm{d}s
\end{equation*}
and
\begin{equation*}
\mathfrak{S}_3=\frac{1}{4\varepsilon} \int_{t_1}^{t_2} \Bigl[|t_2-\theta + \varepsilon|^{2H}+
|\theta - t_1 + \varepsilon|^{2H}-
|t_2-\theta - \varepsilon|^{2H}-
|\theta-t_1-\varepsilon|^{2H}\Bigr]\mathrm{d}\theta\,.
\end{equation*}
We will show that both $\mathfrak{S}_2$ and $\mathfrak{S}_3$ converge to $|t_2-t_1|^{2H}$.\\
\textbf{Step I: Limiting behavior of $\mathfrak{S}_2$}

By a change of variable we can replace the integration interval with $[0,t_2-t_1]$ with the integrand remaining intact. But as the integrand is symmetric in $s$ and $\eta$, we may calculate the integral over a triangular surface hence getting:
\begin{equation*}
\mathfrak{S}_2=\frac{2}{8\varepsilon^2}\int_0^{t_2-t_1} \int_0^{s} \Bigl[|s-\eta+2\varepsilon|^{2H}+|\eta-s+2\varepsilon|^{2H}-
2|s-\eta|^{2H}\Bigr]\mathrm{d}\eta \,\mathrm{d}s\,.
\end{equation*}
By a change of variable of $\gamma=s-\eta$ we get:
\begin{equation}\label{term2}
\mathfrak{S}_2=\frac{1}{4\varepsilon^2}\int_0^{t_2-t_1} \int_0^{s} \Bigl[|\gamma+2\varepsilon|^{2H}+|\gamma-2\varepsilon|^{2H}-
2|\gamma|^{2H}\Bigr]\mathrm{d}\gamma \,\mathrm{d}s\,.
\end{equation}
We will show that $\mathfrak{S}_2$ converges to $|t_2-t_1|^{2H}$ with the following rate of convergence for $H<\frac{1}{2}$
\begin{equation}\label{difference bound for H smaller than half}
\left|\mathfrak{S}_2-|t_2-t_1|^{2H}\right|\leq 4(2\varepsilon)^{2H}
\end{equation}
and
\begin{equation}\label{difference bound for H larger than half}
\left|\mathfrak{S}_2-|t_2-t_1|^{2H}\right|\leq C \varepsilon
\end{equation}
for $H>\frac{1}{2}$. Here $C$ is some constant depending only on $T$ and $H$ .
For the simplicity of notation let $t:=t_2-t_1$. Defining $g(s):=\int_0^s |r|^{2 H}\mathrm{d}r$, \eqref{term2} can be written as:
\begin{equation}\label{first piece 1}
\mathfrak{S}_2=\frac{1}{4\varepsilon^2}\int_0^t\left[g(s+2\varepsilon)+
g(s-2\varepsilon)-2g(s)\right]\mathrm{d}s\,.
\end{equation}
As $g'$ is continuous everywhere and $g''(r)=2H\,\sgn(r) |r|^{2H-1}$ is continuous everywhere except for the origin when $H<\frac{1}{2}$ and everywhere when $H \geq \frac{1}{2}$, this equation can be written as:
\begin{equation}\label{first piece 2}
\mathfrak{S}_2=\frac{1}{4}\int_{-1}^1\int_{-1}^1\int_0^t g''(s+\xi \varepsilon+\eta \varepsilon)\mathrm{d}s\,\mathrm{d}\xi\,\mathrm{d}\eta\,.
\end{equation}

Let $\Delta:=\xi\varepsilon+\eta\varepsilon$ and first suppose that $H<\frac{1}{2}$.\\
Case i) $\Delta \geq 0$:
\begin{equation*}
    \begin{split}
        \left| \int_0^t \bigl( g''(s+\Delta)-2Hs^{2H-1}\bigr)\mathrm{d}s\right|&=2H\int_0^t \bigl(s^{2H-1}-{(s+\Delta)}^{2H-1}\bigr)\mathrm{d}s\\
        & = \left[t^{2H} - (t+\Delta)^{2H}\right]+\Delta^{2H} \leq \Delta^{2H}\,.
    \end{split}
\end{equation*}
Case ii) $-t < \Delta < 0$: 
\begin{equation}\label{case ii}
\begin{split}
\int_0^t \bigl( g''(s+\Delta)-2Hs^{2H-1}\bigr)\mathrm{d}s=
&-2H \int_0^{-\Delta}\bigl((-s-\Delta)^{2H-1}+s^{2H-1}\bigr)\mathrm{d}s\\
&\quad +2H \int_{-\Delta}^t \bigl( (s+\Delta)^{2H-1} - s^{2H-1} \bigr)\mathrm{d}s\,.
\end{split}
\end{equation}
The first term equals $-2|\Delta|^{2H}$ and the second term equals $(t+\Delta)^{2H}-t^{2H}+\Delta^{2H}$ which is bounded by $2|\Delta|^{2H}$.\\
Case iii) $\Delta \leq -t$:
\begin{equation}\label{case iii}
    \begin{split}
        &\left| \int_0^t \bigl( g''(s+\Delta)-2Hs^{2H-1}\bigr)\mathrm{d}s\right|
         = 2H \int_0^{t}\bigl((-s-\Delta)^{2H-1}+s^{2H-1}\bigr)\mathrm{d}s\\
        & \qquad \qquad \qquad \leq 2H \int_0^{-\Delta}\bigl((-s-\Delta)^{2H-1}+s^{2H-1}\bigr)\mathrm{d}s= 2|\Delta|^{2H}\,.
    \end{split}
\end{equation}
Noting that $|\Delta|<2\varepsilon$, inequality \eqref{difference bound for H smaller than half} is proved.

Now we consider the case of $H\geq\frac{1}{2}$. \\
Case i) $\Delta\geq 0$:
\begin{equation}\label{case i}
\begin{split}
        \int_0^t \bigl( g''(s+\Delta)-2Hs^{2H-1}\bigr)\mathrm{d}s
        &=2H\int_0^t \bigl({(s+\Delta)}^{2H-1}-s^{2H-1}\bigr)\mathrm{d}s\\
        &=2H\int_0^t \int_0^{\Delta}(2H-1)(s+\alpha)^{2H-2}\mathrm{d}\alpha\,\mathrm{d}s\\
        &=2H \int_0^{\Delta} \bigl( (t+\alpha)^{2H-1}-\alpha^{2H-1}\bigr)\mathrm{d}\alpha\,.
\end{split}
\end{equation}
As $2H-1<1$ we have $(t+\alpha)^{2H-1}-\alpha^{2H-1}\leq t^{2H-1}$ which shows that the above integral is bounded by $2H t^{2H-1} |\Delta| $ and hence by $2H T^{2H-1} |\Delta| $.\\
Case ii) $-t < \Delta < 0$: Equation \eqref{case ii} remains valid with its first term bounded by $2|\Delta|^{2H}$ which is smaller than $2|\Delta|$, assuming $|\Delta|<1$. As $2H-1>0$, the absolute value of the second term equals:
\begin{equation*}
\begin{split}
&2H \int_{-\Delta}^t \bigl(  s^{2H-1} - (s+\Delta)^{2H-1} \bigr)\mathrm{d}s
= 2H \int_\Delta^0 \int_{-\Delta}^t  (s+\alpha)^{2H-2}(2H-1) \mathrm{d}s \, \mathrm{d}\alpha \\
&\qquad \qquad \qquad = 2H \int_{\Delta}^0 \bigl[  (\alpha+t)^{2H-1} - (-\Delta + \alpha)^{2H-1} \bigr]\mathrm{d}\alpha\\
& \qquad \qquad \qquad \leq 2H \int_{\Delta}^0 (t+\Delta)^{2H-1} \leq 2H  t^{2H-1} |\Delta|\leq 2H  T^{2H-1} |\Delta|\,.
\end{split}
\end{equation*}
The last inequality is true because $2H-1<1$. So we get the bound $(2+2H  T^{2H-1}) |\Delta|$.\\
Case iii) $\Delta \leq -t$: Equation \eqref{case iii} works without any change and we get the bound $2|\Delta|^{2H} \leq 2|\Delta|$.

Noting $|\Delta|\leq 2\varepsilon$ the proof of inequality  \eqref{difference bound for H larger than half} is complete with $C=2^{2H}(2+2H  T^{2H-1})$.\\

In the $H\geq\frac{1}{2}$ regime we can establish the following alternative bound which will be used in section \ref{convergence of u_epsilon}
\begin{equation}\label{alternative difference bound for H larger than half}
\left|\mathfrak{S}_2-|t_2-t_1|^{2H}\right|\leq 2|t_2-t_1| (2H+1) \varepsilon^{2H-1}\,.
\end{equation}

It is shown case by case
\begin{itemize}
  \item For case i), using the first equality in equation \eqref{case i} and noting $(s+\Delta)^{2H-1}-s^{2H-1}\leq \Delta^{2H-1}$ we have the bound $2 H t \Delta^{2H-1}$.
  \item For case ii), the second term on the right hand side in \eqref{case ii} can be bounded by $2H(t-|\Delta|)|\Delta|^{2H-1} \leq 2H t |\Delta|^{2H-1}$ and the first term by $2|\Delta|^{2H} \leq 2 t |\Delta|^{2H-1}$.
  \item In case iii),  using the first equality in \eqref{case iii} it can be bounded by $4H t |\Delta|^{2H-1} $.
\end{itemize}
So we have the bound $2t (2H+1) |\Delta|^{2H-1} \leq 2t (2H+1) \varepsilon^{2H-1}$.\\
\textbf{Step II: Limiting behavior of $\mathfrak{S}_3$}

By setting $t:= t_2-t_1$ and two changes of variables, $\mathfrak{S}_3$ can be written as
\begin{equation*}
\frac{2}{4\varepsilon} \int_{0}^{t} \Bigl(|\theta + \varepsilon|^{2H}-
|\theta - \varepsilon|^{2H}\Bigr)\mathrm{d}\theta
           =\frac{1}{2\varepsilon} \int_{0}^{t} \int_{-\varepsilon}^{+\varepsilon} 2H \, |\theta + \alpha|^{2H-1} \mathrm{d}\alpha \,\mathrm{d}\theta\,.
\end{equation*}
So
\begin{equation}\label{3}
(\mathfrak{S}_3-t^{2H}) =
\frac{1}{2\varepsilon} \int_{-\varepsilon}^{+\varepsilon} \int_{0}^{t}  2H \bigl( |\theta + \alpha|^{2H-1}- \theta^{2H-1}\bigr) \mathrm{d}\theta \,\mathrm{d}\alpha\,.
\end{equation}

Let's first assume $\varepsilon\leq t$. Let's break this integral into three sub-integrals:
\begin{equation*}
\int_{0}^{+\varepsilon} \int_{0}^{t}\dotsb
+ \int_{-\varepsilon}^{0} \int_{0}^{-\alpha}\dotsb
+ \int_{-\varepsilon}^{0} \int_{-\alpha}^{t}\dotsb
\end{equation*}
and call them $A$, $B$ and $C$, respectively.\\
We bound these terms separately for $H \leq \frac{1}{2}$ and $H > \frac{1}{2}$.

First suppose $H \leq \frac{1}{2}$.
\begin{equation}\label{A when H less than half}
\begin{split}
|A| & = \frac{1}{2\varepsilon}  \int_{0}^{+\varepsilon} \int_{0}^{t}   2H \bigl[ \theta^{2H-1} - (\theta + \alpha)^{2H-1}\bigr] \mathrm{d}\theta \,\mathrm{d}\alpha \\
& = \frac{1}{2\varepsilon}  \int_{0}^{+\varepsilon} \bigl[ \alpha^{2H} - (\alpha + t)^{2H} + t^{2H}\bigr] \mathrm{d}\alpha \\
& \leq \frac{1}{2\varepsilon}  \int_{0}^{+\varepsilon} \alpha^{2H} \mathrm{d}\alpha \;
= \frac{1}{2(2H+1)}\varepsilon^{2H}\,.
\end{split}
\end{equation}
For the second term we have
\begin{equation*}
\begin{split}\label{B}
|B| & \leq \frac{1}{2\varepsilon}  \int_{-\varepsilon}^{0} \int_{0}^{-\alpha}   2H \bigl[ \theta^{2H-1} +(-\theta - \alpha)^{2H-1}\bigr] \mathrm{d}\theta \,\mathrm{d}\alpha \\
& = \frac{1}{\varepsilon}  \int_{-\varepsilon}^{0} (-\alpha)^{2H} \mathrm{d}\alpha = \frac{1}{2H+1}\varepsilon^{2H}\,.
\end{split}
\end{equation*}
Finally:
\begin{equation*}
\begin{split}
|C| & = \frac{1}{2\varepsilon}  \int_{-\varepsilon}^{0} \int_{-\alpha}^{t}   2H \bigl[ (\theta + \alpha)^{2H-1} - \theta^{2H-1}\bigr] \mathrm{d}\theta \,\mathrm{d}\alpha \\
& = \frac{1}{2\varepsilon}  \int_{-\varepsilon}^{0} \bigl[(t+\alpha)^{2H}-t^{2H}+(-\alpha)^{2H}\bigr] \mathrm{d}\alpha \\
& \leq \frac{1}{2\varepsilon}  \int_{-\varepsilon}^{0} (-\alpha)^{2H} \mathrm{d}\alpha \;
= \frac{1}{2(2H+1)}\varepsilon^{2H}\,.
\end{split}
\end{equation*}
So for $H \leq \frac{1}{2}$:
\begin{equation*}
|\mathfrak{S}_3-t^{2H}|\leq \frac{2}{2H+1}\varepsilon^{2H}.
\end{equation*}

Now for $H > \frac{1}{2}$: we again examine each of the terms:
\begin{equation}\label{A when H larger than half}
\begin{split}
|A| & = \frac{1}{2\varepsilon}  \int_{0}^{+\varepsilon} \int_{0}^{t}   2H \bigl[ (\theta + \alpha)^{2H-1} - \theta^{2H-1}\bigr] \mathrm{d}\theta \,\mathrm{d}\alpha \\
& = \frac{H}{\varepsilon}  \int_{0}^{+\varepsilon} \int_{0}^{t} \int_{0}^{\alpha}  (2H-1) (\theta + \xi)^{2H-2}\mathrm{d}\xi \,\mathrm{d}\theta \,\mathrm{d}\alpha \\
& = \frac{H}{\varepsilon}  \int_{0}^{+\varepsilon} \int_{0}^{\alpha} \bigl[ (t + \xi)^{2H-1} - {\xi}^{2H-1}\bigr]\mathrm{d}\xi \,\mathrm{d}\alpha \\
& \leq \frac{H}{\varepsilon}  \int_{0}^{+\varepsilon} \int_{0}^{\alpha}t^{2H-1} \mathrm{d}\xi \,\mathrm{d}\alpha \; = \frac{1}{2}H t^{2H-1}\varepsilon\,.
\end{split}
\end{equation}
As equation \eqref{B} remains valid for $H > \frac{1}{2}$, we have:
\begin{equation*}
|B|\leq\frac{1}{2H+1}\varepsilon^{2H} \leq \frac{1}{2H+1}\varepsilon\,.
\end{equation*}
For $|C|$ we use the same trick as in \eqref{A when H larger than half}:
\begin{equation}\label{C when H larger than half}
\begin{split}
|C|
&= \frac{1}{2\varepsilon} \int_{-\varepsilon}^{0} \int_{-\alpha}^{t}
2H \bigl[\theta^{2H-1}-(\theta + \alpha)^{2H-1}\bigr] \mathrm{d}\theta \,\mathrm{d}\alpha \\
&= \frac{H}{\varepsilon} \int_{-\varepsilon}^{0} \int_0^{-\alpha} \int_{-\alpha}^{t}
(2H-1)(\theta + \xi)^{2H-2} \, \mathrm{d}\theta \,\mathrm{d}\xi \,\mathrm{d}\alpha \\
&= \frac{H}{\varepsilon} \int_{-\varepsilon}^{0} \int_0^{-\alpha}
\bigl[(t+\xi)^{2H-1}-(\xi -\alpha)^{2H-1}\bigr] \mathrm{d}\xi \,\mathrm{d}\alpha \\
&\leq \frac{H}{\varepsilon} \int_{-\varepsilon}^{0} \int_0^{-\alpha}
(t +\alpha)^{2H-1} \mathrm{d}\xi \,\mathrm{d}\alpha \\
&\leq \frac{H}{\varepsilon} \int_{-\varepsilon}^{0} \int_0^{-\alpha}
t ^{2H-1} \mathrm{d}\xi \,\mathrm{d}\alpha \; = \frac{1}{2}H t^{2H-1}\varepsilon\,.
\end{split}
\end{equation}
Now we address the case where $\varepsilon > t$. Here we need to break the integral in \eqref{3} into three sub-integrals:
\begin{equation*}
\int_{0}^{+\varepsilon} \int_{0}^{t}\dotsb
+ \int_{-t}^{0} \int_{0}^{-\alpha}\dotsb
+ \int_{-t}^{0} \int_{-\alpha}^{t}\dotsb
+ \int_{-\varepsilon}^{-t} \int_{0}^{t}\dotsb
\end{equation*}
Let's call the terms as $A'$, $B'$, $C'$, $D'$, respectively.\\

One can check easily that the same procedures used for bounding $A$ and $C$ work for $A'$ and $C'$. For $B'$ and $D'$ we have
\begin{equation*}
|B'| \leq \frac{1}{2\varepsilon}  \int_{-t}^{0} \int_{0}^{-\alpha}   2H \bigl[ \theta^{2H-1} +(-\theta - \alpha)^{2H-1}\bigr] \mathrm{d}\theta \,\mathrm{d}\alpha\,,
\end{equation*}
and
\begin{equation*}
\begin{split}
|D'| & \leq \frac{1}{2\varepsilon}  \int_{-\varepsilon}^{-t} \int_{0}^{t}   2H \bigl[ \theta^{2H-1} +(-\theta - \alpha)^{2H-1}\bigr] \mathrm{d}\theta \,\mathrm{d}\alpha \\
& \leq \frac{1}{2\varepsilon}  \int_{-\varepsilon}^{-t} \int_{0}^{-\alpha}   2H \bigl[ \theta^{2H-1} +(-\theta - \alpha)^{2H-1}\bigr] \mathrm{d}\theta \,\mathrm{d}\alpha\,.
\end{split}
\end{equation*}
Hence
\begin{equation*}
|B'|+|D'| \leq |B|\,.
\end{equation*}

So in brief the same bounds found above for $|\mathfrak{S}_3-t^{2H}|$ for the case $\varepsilon\leq t$ remain valid for the case $\varepsilon > t$ too. So inequality \eqref{piecewise} is proved.

Now we turn back to the proof of proposition \ref{main variance inequality}. we have:
\begin{equation*}
\begin{split}
&\mathbb{E}\Bigl|\int_0^t \dot{W_{\varepsilon}}\bigl(s,X(s)\bigr)\mathrm{d}s - \int _0^t W\bigl(\mathrm{d}s, X(s)\bigr)\Bigr|^2\\
& \qquad \qquad \leq \mathbb{E}\Bigl\{\Bigl(\sum_{i=0}^N\Bigl|W(t_{i+1}) - W(t_i) -\int_{t_i}^{t_{i+1}} \dot{W_{\varepsilon}}(\theta)\mathrm{d}\theta \Bigr|\Bigl)^2\Bigl\}\\
& \qquad \qquad \leq C_1 (N+1)^2 \varepsilon^{\min\{2H,1\}}
\leq C_2 N^2 \varepsilon^{\min\{2H,1\}}\,.
\end{split}
\end{equation*}
\end{proof}

\section{Convergence of $u_\varepsilon$ }\label{convergence of u_epsilon}
In this section, using simple random walk properties we prove that $\tilde{u}_\varepsilon$ and its Malliavin derivative both converge to zero in $\mathcal{L}^2$.
\begin{prop}\label{theorem on convergence of u_epsilon}
$\tilde{u}_\varepsilon:=u_\varepsilon-u$ converges to $0$ in $\mathbb{D}^{1,2}$ uniformly in $[0,T]$, i.e.
\begin{equation*}
\sup_{s\in[0,T]} \mathbb{E}\bigl(|\tilde{u}_\varepsilon(s,x)|^2+\|\nabla \tilde{u}_\varepsilon(s,x)\|_{\mathcal{H}}^2\bigr) \longrightarrow 0 \qquad \text{as} \quad \epsilon \downarrow 0\,.
\end{equation*}
\end{prop}

Let $X:[0,T]\rightarrow \mathds{Z}^d$ be a piecewise constant function on the lattice $\mathds{Z}^d$ with jump times $t_1<t_2< \cdots <t_N$. Let also $t_0:=0$ and $t_{N+1}:=T$. For any given $\delta >0$ we may chop up $[0,T]$ into calm periods and rough ones. A calm period is defined as an interval in which all the consecutive jumps are at least $\delta$ apart, and a rough period as one in which all the consecutive jumps are at most $\delta$ apart. We additionally require that these intervals begin with a jump and end with another.

We also define $R$ as the number of jumps in $[0,T]$ that are within $\delta$ distance of their previous one. In other words, $R$ is defined to be the cardinality of $\{i\, | \; t_{i}-t_{i-1} < \delta, t_{i} \leq T\}$

\begin{lem}
Consider a Poisson process with intensity $\lambda$ and let $R$(=$R_T$) be defined for any sample path of the Poisson process as above. Then for any given $\delta>0$, we have
\begin{equation*}
\mathrm{P}(R\geq n)\leq (C \delta)^n\,,
\end{equation*}
where $C$ is a constant that depends only on $T$ and $\lambda$.
\end{lem}
\begin{proof}
Let $A$ be the event of having at least one jump in $[0,t]$ which is within $\delta$ of a previous one and $B$ be the event of having at least one jump in $[0,\delta]$. Let also $N(t)$ be the number of jumps in $[0,t]$ and $t_0:=0$. We have
\begin{equation*}
\begin{split}
\mathrm{P}(A \cup B) &\leq \sum_{k=1}^\infty \mathrm{P}(t_k-t_{k-1}<\delta \;\text{and}\; t_{k-1} <t)\\
&= \sum_{k=1}^\infty \mathrm{P}(t_k-t_{k-1}<\delta \;|\; t_{k-1} <t) \,\mathrm{P}(t_{k-1} <t)\\
&= (1-e^{-\lambda \delta})\sum_{k=1}^\infty \mathrm{P}(t_{k-1} <t)\\
&=(1-e^{-\lambda \delta})\sum_{k=0}^\infty \mathrm{P}(N(t)\geq k)\\
&=(1-e^{-\lambda \delta})\bigl( \mathbb{E}\left( N(t) \right)+1\bigr)\,.
\end{split}
\end{equation*}
Using the fact that the expectation of $N(t)$ is $\lambda t$ and noting the inequality $1-e^{-\lambda \delta} \leq \lambda \delta$, we get $\mathrm{P}(A \cup B) \leq C_t \delta$, where $C_t=\lambda \delta (1+ t\lambda)$. In particular $C_t$ is increasing in $t$.

Now we define $\sigma_1$ as the first jump time that is within $\delta$ of the previous one, i.e. $\sigma_{1}:=\inf\{t_k > 0 \; ; \; t_k-t_{k-1}< \delta\}$. Having defined $\sigma_n$ we define $\sigma_{n+1}$ as the first jump time after $\sigma_n$ that is within $\delta$ of the previous one, i.e. $\sigma_{n+1}:=\inf\{t_k > \sigma_{n} \; ; \; t_k-t_{k-1}< \delta\}$. We have
\begin{equation*}
\mathrm{P}(\sigma_{i+1}<T \, | \; \sigma_{i}) \leq
\begin{cases}
0 & \text{if} \; \; \sigma_i \geq T \\
C_{T-\sigma_i} & \text{if} \; \; \sigma_i < T\,.
\end{cases}
\end{equation*}
As $C_t$ is an increasing function in $t$ we have the following uniform bound:
\begin{equation*}
\mathrm{P}(\sigma_{i+1}<T \, | \;\sigma_{i}) \leq (C_{T} \,\delta) \mathbf{1}_{\{\sigma_i < T\}}\,.
\end{equation*}
So
\begin{equation*}
\mathrm{P}(\sigma_{i+1}<T)=\mathbb{E}\bigl[\mathrm{P}(\sigma_{i+1}<T \, | \;\sigma_{i})\bigr] \leq (C_{T} \delta) \mathrm{P}(\sigma_i < T)\,.
\end{equation*}
So by induction
\begin{equation*}
\mathrm{P}(\sigma_{k}<T) \leq  (C_{T} \delta)^k\,.
\end{equation*}
Now noticing that $R \geq n$ implies $\sigma_n < T$, we get
\begin{equation*}
\mathrm{P}(R \geq n) \leq \mathrm{P}(\sigma_{n}<T) \leq  (C_{T} \delta)^n\,.
\end{equation*}
\end{proof}

\begin{lem}\label{rough length bound}
For a Poisson process of intensity $\lambda$ and for any given $\delta >0$, let $L$ be the total length of its rough periods in $[0,T]$ and $K$ be the number of rough periods in $[0,T]$. Then there exists a constant $C$ depending only on $T$ and $\lambda$ such that
\begin{equation*}
\mathrm{P}(K \geq n) \leq (C\delta)^n
\end{equation*}
and
\begin{equation*}
\mathrm{P}(L \geq n\delta) \leq (C\delta)^n.
\end{equation*}
\end{lem}
\begin{proof}
As $L < R\delta$ and $K \leq R$, any of $L \geq n\delta$ or $K \geq n$ implies $R \geq n$. The result follows from the previous lemma.
\end{proof}
 Now we are ready to prove the following lemma.
\begin{lem}\label{uniform boundedness}
For any $p \geq 1$, there exists $M>0$ such that $E|u_\varepsilon(t,x)|^p$ is bounded uniformly in $(\varepsilon,t,x)\in(0,M]\times[0,T]\times\mathds{Z}^d$. $E|u(t,x)|^p$ is also bounded uniformly in $(t,x)\in [0,T]\times\mathds{Z}^d$.
\end{lem}
\begin{proof}
First consider $E|u(t,x)|^p$.
\begin{equation*}
\begin{split}
&E|u(t,x)|^p
\leq \|u_o\|_\infty^p \mathbb{E}^x \, \mathbb{E} \, \exp\bigl[p \int_{0}^{t} W\bigl(\mathrm{d}s, X(t-s)\bigr) \bigr]\\
&\qquad = \|u_o\|_\infty^p \mathbb{E}^x \exp \Bigl( \frac{p^2}{2}\,\var \bigl[\int_{0}^{t} W\bigl(\mathrm{d}s, X(t-s)\bigr)\bigr]\Bigr)\,.
\end{split}
\end{equation*}
So it is enough to find a uniform bound on $\var \bigl[\int_{0}^{t} W\bigl(\mathrm{d}s, X(t-s)\bigr)\bigr]$. For any sample path $X(\cdot)$ of simple random walk on $\mathds{Z}^d$ let $t_1< t_2 < \dotsb < t_N$ be the jump times of the reversed path $X(t-\cdot)$ and $x_1$, $x_2$, ..., $x_{N+1}$ be its values. Let also $t_0:=0$ and $t_{N+1}:=t$. We have
\begin{equation*}
\begin{split}
\var \bigl[\int_{0}^{t} W\bigl(\mathrm{d}s, X(t-s)\bigr)\bigr]
&= \var \bigl[\sum_{i=1}^{N+1}\int_{t_{i-1}}^{t_i} W\bigl(\mathrm{d}s, x_i\bigr)\bigr]\\
&= \var \bigl[\sum_{i=1}^{N+1} \; W(t_i, x_i)-W(t_{i-1}, x_i) \bigr]\,.
\end{split}
\end{equation*}

For $H\geq \frac{1}{2}$ we have
\begin{equation*}
\begin{split}
&\var \bigl[\sum_{i=1}^{N+1} \; W(t_i, x_i)-W(t_{i-1}, x_i) \bigr]\\
&\qquad \leq  (N+1) \sum_{i=1}^{N+1} \; \var\bigl[ W(t_i, x_i)-W(t_{i-1}, x_i) \bigr]\\
&\qquad = (N+1) \sum_{i=1}^{N+1} \; (t_i-t_{i-1})^{2H}
\leq (N+1) t^{2H}\,.
\end{split}
\end{equation*}
As $N$ is a Poisson random variable, $\mathbb{E} \exp(C N)$ is finite for any constant $C$.

For $H\leq \frac{1}{2}$ we use the well-known property that disjoint increments of a fractional Brownian motion with Hurst parameter less than half are negatively correlated. So we have
\begin{equation*}
\begin{split}
&\var \bigl[\sum_{i=1}^{N+1} \; W(t_i, x_i)-W(t_{i-1}, x_i) \bigr]
\leq  \sum_{i=1}^{N+1} \; \var\bigl[ W(t_i, x_i)-W(t_{i-1}, x_i) \bigr]\\
&\qquad \qquad \qquad \qquad = \sum_{i=1}^{N+1} \; (t_i-t_{i-1})^{2H}
\leq (N+1)^{1-2H} t^{2H}\,.
\end{split}
\end{equation*}
In the last inequality we have used the fact that for $H\leq \frac{1}{2}$, the expression $x_1^{2H} + x_2^{2H} + \dotsb + x_m^{2H}$ achieves its maximum when all $x_i$'s are equal and the maximum is hence $m^{1-2H}(\sum_i x_i)^{2H}$.\\
Again as $N$ is Poisson, $\mathbb{E} \exp(C N^{\alpha})$ is finite for any constants $C$ and $\alpha \leq 1$.

Now let us consider $E|u_\varepsilon(t,x)|^p$:
\begin{equation}\label{boundedness of exponential}
\begin{split}
E|u_\varepsilon(t,x)|^p
&\leq \|u_o\|_\infty^p \mathbb{E}^x \, \mathbb{E} \, \exp\bigl[p\int_{0}^{t} \dot{ W_{\varepsilon}}\bigl(s, X(t-s)\Bigr)\mathrm{d}s\bigr] \\
&= \|u_o\|_\infty^p \mathbb{E}^x \, \exp \Bigl( \frac{p^2}{2}\,\var \bigl[\int_{0}^{t} \dot{W_{\varepsilon}}\bigl(s, X(t-s)\Bigr)\mathrm{d}s\bigr]\Bigr)\,.
\end{split}
\end{equation}
Again we need to distinguish between $H$ larger and less than half.\\

When $H$ is larger than a half, $\var\left( \int_{t_1}^{t_2} \dot{W_{\varepsilon}}(s)\mathrm{d}s\right)$ being equal to $\mathfrak{S}_2$ introduced in section \ref{Approximation rate}, is bounded by $(t_2-t_1)^{2H}+2(t_2-t_1) (2H+1) \varepsilon^{2H-1}$ by inequality \eqref{alternative difference bound for H larger than half}. With the above notation
\begin{equation*}
\begin{aligned}
&\var \bigl[\int_{0}^{t} \dot{W_{\varepsilon}}\bigl(s, X(t-s)\Bigr)\mathrm{d}s\bigr]
= \var \bigl[\sum_{i=1}^{N+1} \int_{t_{i-1}}^{t_i} \dot{W_{\varepsilon}}(s,x_i)\mathrm{d}s \bigr]\\
&\qquad \qquad \leq (N+1) \sum_{i=1}^{N+1} \var \bigl(\int_{t_{i-1}}^{t_i} \dot{W_{\varepsilon}}(s,x_i)\mathrm{d}s\bigr)\\
&\qquad \qquad \leq (N+1) \sum_{i=1}^{N+1} \Bigl( (t_{i+1}-t_{i})^{2H}+2(t_{i+1}-t_{i}) (2H+1) \varepsilon^{2H-1}\Bigr)\\
&\qquad \qquad \leq (N+1)\Bigl(t^{2H}+2(2H+1)\varepsilon^{2H-1} t\Bigr)\,.
\end{aligned}
\end{equation*}
Again we get a multiple of $N$ and hence a finite bound.

When $H\leq \frac{1}{2}$, the situation is more complicated.
Let $\{t_i\}_{i=1}^N$ be the increasingly ordered jump times of $\{X(t-s)\,;\, s\in [0,t]\}$ with additional convention of $t_0:=0$ and $t_{N+1}:=t$. We decompose $[0,t]$ into calm and rough periods of $X(t-\cdot)$ with respect to $\delta=2\varepsilon$. Let increasingly enumerate the set of indices $\{i\;;\;t_i-t_{i-1}\geq \delta\}$ as $\{t_{i_k}\}_k$. In other words, we single out and enumerate those time intervals $[t_i-1,t_i]$ whose length is larger than or equal to $\delta=2\varepsilon$. It is evident that such intervals constitute the calm periods. Let also $\{Y_k\}_k$ be the integral of $\dot{W_{\varepsilon}}(\cdot,x_{i_k})$ over the time interval $[t_{i_k-1},t_{i_k}]$, i.e. $Y_k:=\int_{t_{i_{k}-1}}^{t_{i_k}} \dot{W_{\varepsilon}}(s,x_{i_k})\mathrm{d}s$. Let also $Z$ be the sum of the integrals over all rough periods. Using equation \eqref{boundedness of exponential}, Cauchy-Schwartz and the simple inequality $\mathbb{E}(X+Y)^2 \leq 2\mathbb{E} X^2 +2\mathbb{E} Y^2$, we have
\begin{equation*}
\begin{split}
&E|u_\varepsilon(t,x)|^p
\leq \|u_o\|_\infty^p \mathbb{E}^x \, \exp \bigl( \frac{p^2}{2}\,\mathbb{E} (Z +\sum_k Y_k)^2 \, \bigr)\\
&\qquad \leq \|u_o\|_\infty^p \bigl[ \mathbb{E}^x \, \exp \bigl( 2p^2\,\mathbb{E} (Z^2) \,\bigr)\bigr]^{1/2} \bigl[ \mathbb{E}^x \, \exp \bigl( 2p^2\,\mathbb{E} \, (\sum_k Y_k)^2 \,\bigr)\bigr]^{1/2}.
\end{split}
\end{equation*}
Once again we will use the negativeness of the covariance of disjoint increments of a fractional Brownian motion with Hurst parameter less than half.

First we consider the integral over the rough periods, i.e. the first term above. Let $I$ be the union of all the rough intervals in $[0,t]$.

We notice that for $\alpha, \beta \in [0,t]$, and a fractional Brownian motion $W(\cdot)$ of Hurst parameter $H\leq1/2$ we have
$$
\mathbb{E}\dot{W_{\varepsilon}}(\alpha)\dot{W_{\varepsilon}}(\beta)\leq0 \qquad \text{for}\quad |\alpha-\beta|\geq 2\varepsilon\,,
$$
which is nothing but the negative correlation of non-overlapping increments of a fBM,
and
$$
\bigl|\mathbb{E}\dot{W_{\varepsilon}}(\alpha)\dot{W_{\varepsilon}}(\beta)\bigr|\leq \frac{4(4\varepsilon)^{2H}}{(2\varepsilon)^2}\qquad \text{for} \quad |\alpha-\beta|< 2\varepsilon\,,
$$
which is easily followed by a simple calculation.

This shows that for $\alpha, \beta \in [0,t]$, there are only two possibilities: either
$\dot{W_{\varepsilon}}\bigl(\alpha,X(t-\alpha)\bigr)$ and $\dot{W_{\varepsilon}}\bigl(\beta,X(t-\beta)\bigr)$ have negative correlation or they are uncorrelated, depending on whether $X(t-\alpha)$ is the same as $X(t-\beta)$ or not. So we have
\begin{equation*}
\begin{split}
\mathbb{E} (Z^2)
&= \mathbb{E} \bigl[ \int_{I} \dot{W_{\varepsilon}}\bigl(\alpha,X(t-\alpha)\bigr)\mathrm{d}\alpha \int_{I} \dot{W_{\varepsilon}}\bigl(\beta,X(t-\beta)\bigr)\mathrm{d}\beta \bigl]\\
&= \int_{\alpha\in I} \int_{\beta \in I} \mathbb{E} \bigl[ \dot{W_{\varepsilon}}\bigl(\alpha,X(t-\alpha)\bigr)  \dot{W_{\varepsilon}}\bigl(\beta,X(t-\beta)\bigr) \bigl]\mathrm{d}\beta \mathrm{d}\alpha \\
&\leq \int_{\alpha\in I} \int_{\beta \in I} \mathbb{E} \bigl[ \dot{W_{\varepsilon}}\bigl(\alpha,X(t-\alpha)\bigr)  \dot{W_{\varepsilon}}\bigl(\beta,X(t-\beta)\bigr) \bigl] \mathbf{1}_{|\alpha-\beta|< 2\varepsilon} \mathrm{d}\beta \mathrm{d}\alpha\\
&\leq \int_{\alpha\in I} \int_{\beta \in I} \bigl|\mathbb{E} \bigl( \dot{W_{\varepsilon}}(\alpha)  \dot{W_{\varepsilon}}(\beta) \bigl)\bigr| \mathbf{1}_{|\alpha-\beta|< 2\varepsilon} \mathrm{d}\beta \mathrm{d}\alpha\\
&\leq \int_{\alpha\in I} \int_{\beta \in I} \frac{2\varepsilon^{2H}}{\varepsilon^2}  \mathbf{1}_{|\alpha-\beta|< 2\varepsilon} \mathrm{d}\beta \mathrm{d}\alpha \\
&=  \frac{2\varepsilon^{2H}}{\varepsilon^2} \int_{\alpha \in I} (4\varepsilon) \mathrm{d}\alpha \, \leq \, 8 \varepsilon^{2H-1} L\,,
\end{split}
\end{equation*}
where $L$ is the total length of rough periods, i.e. the length of $I$.\\
So
\begin{equation*}
\mathbb{E}^x \, \exp \bigl( 2p^2\,\mathbb{E} (Z^2) \,\bigr) \leq \mathbb{E}^x \,\exp \bigl( 16p^2 \varepsilon^{2H} L / \varepsilon \,\bigr)\,.
\end{equation*}
As $L / \varepsilon$ has exponential tail by lemma \ref{rough length bound}, the above expectation is finite for $\varepsilon$ small enough.

For the second term, $\mathbb{E} \, (\sum_k Y_k)^2$, observe that the length of each time interval $[t_{i_k-1},t_{i_k}]$ is larger than $2 \varepsilon$ which means the distance of every two non-neighboring such intervals is at least $2 \varepsilon$. But this means that only consecutive $Y_k$'s can be positively correlated because for any two intervals $I_1$ and $I_2$ that are at least $2\varepsilon$ apart, the integrals $\int_{I_1} \dot{W_{\varepsilon}}(s)\mathrm{d}s$ and $\int_{I_2} \dot{W_{\varepsilon}}(s)\mathrm{d}s$ are negatively correlated which in turn is a consequence of the negative correlation of disjoint intervals of a fractional Brownian motion with $H\leq \frac{1}{2}$. So
\begin{equation*}
\begin{split}
\mathbb{E}\bigl[ (\sum_k Y_k)^2\bigr]
& \leq \mathbb{E}(Y_1^2)+2\mathbb{E}(Y_1 Y_2)+\mathbb{E}(Y_2^2)
+2\mathbb{E}(Y_2 Y_3)+\mathbb{E}(Y_3^2)+...\\
& \qquad +2\mathbb{E}(Y_{n-1} Y_{n})+\mathbb{E}(Y_m^2)\\
& \leq 2\mathbb{E}(Y_1^2)+3\mathbb{E}(Y_2^2)+3\mathbb{E}(Y_3^2)+...+3\mathbb{E}(Y_{n-1}^2)+2\mathbb{E}(Y_m^2)\\
& \leq 3 \sum_k \mathbb{E}(Y_k^2)\,.
\end{split}
\end{equation*}
In the first inequality we have used the fact that for non-consecutive $Y_i$ and $Y_j$, their covariance $\mathbb{E}(Y_i Y_j)$ is negative and in the last inequality we have used $2\mathbb{E}(XY)\leq \mathbb{E}(X^2)+\mathbb{E}(Y^2)$. Using equation \eqref{difference bound for H smaller than half} we have
\begin{equation*}
\var\left[ \int_{t_i}^{t_{i+1}} \dot{W_{\varepsilon}}(s)\mathrm{d}s \right] \leq (t_{i+1}-t_i)^{2H}+ 4(2\varepsilon)^{2H}\,.
\end{equation*}
So noting $m\leq N$, where $N$ denotes the number of jumps in $[0,t]$ and using the fact that $x_1^{2H} + x_2^{2H} + \dotsb + x_m^{2H}$ is bounded by $m^{1-2H}(\sum_i x_i)^{2H}$ for $H\leq \frac{1}{2}$ which is a consequence of concavity of $(\cdot)^{2H}$, we get
$$
\begin{aligned}
\mathbb{E}\bigl[ (\sum_k Y_k)^2\bigr]
&\leq \,3 \sum_{k=1}^{m}[\,(t_{i_k}-t_{{i_k}-1})^{2H}+ 4(2\varepsilon)^{2H}]\\
&\leq 3 m^{1-2H}[\sum_{k=1}^{m}(t_{i_k}-t_{{i_k}-1})]^{2H} + 12m(2\varepsilon)^{2H}\\
&\leq 3 (N+1)^{1-2H} t^{2H} + 12(N+1)(2\varepsilon)^{2H}\,.
\end{aligned}
$$
\end{proof}

\begin{proof}[\textsc{Proof of proposition \ref{theorem on convergence of u_epsilon}}]
We give the same argument used in \cite{Nualart2}.\\
Since $u_o$ is bounded, for simplicity and without any loss of generality we drop it from now on.\\
For $p\geq 1$ arbitrary, using the inequalities $|e^a - e^b|\leq (e^a + e^b)|a - b|$ and $(a+b)^n \leq 2^{n-1} (a^n+b^n)$ and also H\"older's and Jensen's inequalities we get
\begin{equation}\label{bound on suprimum}
\begin{split}
&\mathbb{E}\bigl|u^\varepsilon(t,x) - u(t,x)\bigr|^p \\
& \qquad = \mathbb{E} \bigl| \mathbb{E}^x\bigl(e^{\mathbf{W}(g_{t,x}^{\varepsilon, X})} - e^{\mathbf{W}(g_{t,x}^{X})}\bigr)\bigr|^p \\
& \qquad \leq \mathbb{E}^x \,\mathbb{E}\bigl|e^{\mathbf{W}(g_{t,x}^{\varepsilon, X})} - e^{\mathbf{W}(g_{t,x}^{X})}\bigr|^p \\
& \qquad \leq \mathbb{E}^x \Bigl( \mathbb{E}\bigl(e^{\mathbf{W}(g_{t,x}^{\varepsilon, X})} + e^{\mathbf{W}(g_{t,x}^{X})}\bigr)^{2p} \Bigr)^{1/2} \mathbb{E}^x \Bigl( \mathbb{E} |\mathbf{W}(g_{t,x}^{\varepsilon, X}) - \mathbf{W}(g_{t,x}^{X})|^{2p} \Bigr)^{1/2} \\
& \qquad \leq C \Bigl( \mathbb{E}^x\, \mathbb{E}\bigl(e^{2p \mathbf{W}(g_{t,x}^{\varepsilon, X})} + e^{2p \mathbf{W}(g_{t,x}^{X})}\bigr) \Bigr)^{1/2} \mathbb{E}^x  \, \mathbb{E} |\mathbf{W}(g_{t,x}^{\varepsilon, X}) - \mathbf{W}(g_{t,x}^{X})|^2\,,
\end{split}
\end{equation}
where in the second inequality we have used the fact that for Gaussian random variables all the $n$-norms are equivalent to 2-norm.\\
So by applying lemma \ref{uniform boundedness} and proposition \ref{main variance inequality} we obtain
\begin{equation*}
\sup_{t\in[0,T]} \mathbb{E} |\tilde{u}_\varepsilon(t,x)|^2 \longrightarrow 0 \qquad \text{as} \quad \varepsilon \downarrow 0\,.
\end{equation*}
For the convergence of $\nabla \widetilde{u}_\varepsilon$, we use the fact that for a separably-valued $\mathbb{D}^{1,2}$-valued random variable $f\in \mathcal{L}^1(\mathcal{X}; \mathbb{D}^{1,2})$ with $\mathcal{X}$ a probability space independent of the underlying Gaussian space of $\mathbb{D}^{1,2}$, we have $\mathbb{E} \nabla f = \nabla \mathbb{E} f$ provided that $\mathbb{E}(\|f\|_{\mathbb{D}^{1,2}}) < \infty$, where the expectations are taken with respect to $\mathcal{X}$. This follows from lemma \ref{Commutativity of integration with operators}.\\
So we have
$$
\nabla u_\varepsilon (t,x)=\mathbb{E}^x[ g_{t,x}^{\varepsilon, X} e^{\mathbf{W}(g_{t,x}^{\varepsilon, X})}]
$$
$$
\nabla u(t,x)=\mathbb{E}^x[ g_{t,x}^{X} e^{\mathbf{W}(g_{t,x}^{X})}]\,.
$$
So
\begin{equation*}
\begin{split}
&E\| \nabla u^\varepsilon(t,x)-\nabla u(t,x)\|_\mathcal{H}^2 \\
&\qquad =E\bigl\| \mathbb{E}^x \bigl(g_{t,x}^{\varepsilon, X} e^{\mathbf{W}(g_{t,x}^{\varepsilon, X})} - g_{t,x}^{X} e^{\mathbf{W}(g_{t,x}^{X})}\bigr)\bigr\|_\mathcal{H}^2 \\
&\qquad \leq 2 \mathbb{E} \mathbb{E}^x \Bigl( e^{\mathbf{W}(g_{t,x}^{\varepsilon, X})} \| g_{t,x}^{\varepsilon, X}-g_{t,x}^{X}\|_\mathcal{H}^2  \Bigr) \\
&\qquad \qquad + 2 \mathbb{E} \mathbb{E}^x \Bigl( |e^{\mathbf{W}(g_{t,x}^{\varepsilon, X})}-e^{\mathbf{W}(g_{t,x}^{X})}|^2 \| g_{t,x}^{X}\|_\mathcal{H}^2  \Bigr)\,.
\end{split}
\end{equation*}
If we apply the Schwartz inequality and note that $\| g_{t,x}^{\varepsilon, X}-g_{t,x}^{X}\|_\mathcal{H}^2=E|\mathbf{W}(g_{t,x}^{\varepsilon, X})-\mathbf{W}(g_{t,x}^{X})|^2$,  along with fact that for Gaussian random variables all norms are equivalent to the $2$-norm, using equation \eqref{bound on suprimum}, lemma \ref{uniform boundedness} and proposition \ref{main variance inequality} we get

\begin{equation*}
\sup_{t\in[0,T]} \mathbb{E} \|\nabla \tilde{u}_\varepsilon(t,x)\|_{\mathcal{H}}^2 \longrightarrow 0 \qquad \text{as} \quad \epsilon \downarrow 0\,.
\end{equation*}
\end{proof}

\section{Convergence of $V_{1,\varepsilon}$}\label{convergence of v_1}

For $V_{1,\varepsilon}$ we use basically the same proof as in \cite{Nualart2}. As one can easily show that
$$\int_0^t \| \tilde{u}_\varepsilon(s,x) g^{\varepsilon}_{s,x} \|_{\mathbb{D}^{1,2}(\mathcal{H})} \mathrm{d}s < \infty\,,$$
where $\mathbb{D}^{1,2}(\mathcal{H})$ denotes the Sobolev space of $\mathcal{H}$-valued $\mathcal{L}^2$ random variables with $\mathcal{L}^2$ Malliavin derivatives, we can apply lemma \ref{Commutativity of integration with operators} to get:
\begin{equation*}
V_{1,\varepsilon}=\boldsymbol{\delta}(\psi_\varepsilon)\,,
\end{equation*}
where
\begin{equation*}
\psi_\varepsilon := \int_0^t \tilde{u}_\varepsilon(s,x) g^{\varepsilon}_{s,x}  \mathrm{d}s.
\end{equation*}
So using inequality \eqref{divergence contractivity}, we have
\begin{equation*}
\mathbb{E} \bigl(|V_{1,\varepsilon}|^2\bigr)= \mathbb{E}\bigl( \boldsymbol{\delta} (\psi_\varepsilon)^2\bigr) \leq \mathbb{E}\bigl( \|\psi_\varepsilon\|_\mathcal{H}^2\bigr) + \mathbb{E}\bigl( \|\nabla \psi_\varepsilon\|_{\mathcal{H}\otimes \mathcal{H}}^2\bigr)\,.
\end{equation*}
For the first right hand side term we have
\begin{equation*}
\begin{split}
&\mathbb{E}\bigl( \|\psi_\varepsilon\|_\mathcal{H}^2\bigr)\\
&\quad = \int_0^t \int_0^t \mathbb{E}\bigl(\tilde{u}_\varepsilon(s_1,x)\tilde{u}_\varepsilon(s_2,x)\bigr) \langle g^{\varepsilon}_{s_1,x},g^{\varepsilon}_{s_2,x} \rangle  \mathrm{d}s_1 \mathrm{d}s_2 \\
&\quad \leq M_1 \int_0^t \int_0^t \bigl|\mathbb{E}\bigl(\dot{W_{\varepsilon}}(s_1,x) \dot{W_{\varepsilon}}(s_2,x)\bigr)\bigr| \mathrm{d}s_1 \mathrm{d}s_2\,,
\end{split}
\end{equation*}
where $M_1=\sup_{s\in[0,t]} \mathbb{E} |\tilde{u}_\varepsilon(s,x)|^2$. Here taking the integration out of the inner product is justified by once more using lemma \ref{Commutativity of integration with operators}.\\
$\int_0^t \int_0^t \bigl|\mathbb{E}\bigl(\dot{W_{\varepsilon}}(s_1,x) \dot{W_{\varepsilon}}(s_2,x)\bigr)\bigr| \mathrm{d}s_1 \mathrm{d}s_2$ being the same as the term $\mathfrak{S}_2$ in equation \eqref{piecewise terms}, is uniformly upper-bounded using equations \eqref{difference bound for H smaller than half} and \eqref{difference bound for H larger than half}. On the other hand, $M_1$ goes to zero as $\varepsilon \downarrow 0$. So it follows that $\mathbb{E}\bigr( \|\psi_\varepsilon\|_\mathcal{H}^2\bigl)$ converges to zero.

For the second term, applying lemma \ref{Commutativity of integration with operators} to the derivative operator and inner product we get
\begin{equation*}
\begin{split}
&\mathbb{E}\bigl( \|\nabla \psi_\varepsilon\|_{\mathcal{H}\otimes \mathcal{H}}^2\bigr)\\
&\quad = \mathbb{E} \bigl\langle \nabla \int_0^t \tilde{u}_\varepsilon(s_1,x) g^{\varepsilon}_{s_1,x}  \mathrm{d}s_1 , \nabla \int_0^t \tilde{u}_\varepsilon(s_2,x) g^{\varepsilon}_{s_2,x}  \mathrm{d}s_2 \bigr\rangle\\
&\quad = \mathbb{E} \bigl\langle\int_0^t \nabla\bigl(\tilde{u}_\varepsilon(s_1,x)\bigr) \otimes g^{\varepsilon}_{s_1,x}  \mathrm{d}s_1 , \int_0^t \nabla\bigl(\tilde{u}_\varepsilon(s_2,x)\bigr) \otimes g^{\varepsilon}_{s_2,x}  \mathrm{d}s_2\bigr\rangle \\
&\quad = \mathbb{E} \int_0^t \int_0^t \bigl\langle \nabla\bigl(\tilde{u}_\varepsilon(s_1,x)\bigr) \otimes g^{\varepsilon}_{s_1,x}  , \nabla\bigl(\tilde{u}_\varepsilon(s_2,x)\bigr) \otimes g^{\varepsilon}_{s_2,x}  \bigr\rangle \mathrm{d}s_1 \mathrm{d}s_2\\
&\quad = \int_0^t \int_0^t \mathbb{E} \bigl\langle \nabla\bigl(\tilde{u}_\varepsilon(s_1,x)\bigr),\nabla\bigl(\tilde{u}_\varepsilon(s_2,x)\bigr)  \bigr\rangle \langle g^{\varepsilon}_{s_1,x}  , g^{\varepsilon}_{s_2,x}  \rangle \mathrm{d}s_1 \mathrm{d}s_2\\
&\quad \leq M_2 \int_0^t \int_0^t |\langle g^{\varepsilon}_{s_1,x}  , g^{\varepsilon}_{s_2,x}  \rangle| \mathrm{d}s_1 \mathrm{d}s_2\\
&\quad = M_2 \int_0^t \int_0^t \bigl|\mathbb{E}[\dot{W_{\varepsilon}}(s_1,x) \dot{W_{\varepsilon}}(s_2,x)]\bigr|\,,
\end{split}
\end{equation*}
where $M_2=\sup_{s\in[0,t]} \mathbb{E} \|\nabla \tilde{u}_\varepsilon(s,x)\|_{\mathcal{H}}^2$. \\
The same argument given for the first term above shows that $\mathbb{E}\bigl( \|\nabla \psi_\varepsilon\|_{\mathcal{H}\otimes \mathcal{H}}^2\bigr)$ also converges to zero as $\varepsilon$ goes down to zero.

\section{Convergence of $V_{2,\varepsilon}$}\label{convergence of v_2}
Establishing the convergence of $V_{2,\varepsilon}$ is more involved. First applying lemma \ref{Commutativity of integration with operators} to $u$ and $u_\varepsilon$ for the derivative operator we get
$$
\nabla u_\varepsilon (s,x)=\mathbb{E}^x[ u_o(X(s)) \; e^{\mathbf{W}(g_{s,x}^{\varepsilon, X})} g_{s,x}^{\varepsilon, X}]
$$
and
$$
\nabla u_ (s,x)=\mathbb{E}^x[ u_o(X(s)) \; e^{\mathbf{W}(g_{s,x}^{X})} g_{s,x}^{X}]\,.
$$
Let
\begin{equation*}
A^X(s,x):=u_o(X(t)) \; e^{\mathbf{W}(g_{s,x}^{X})}
\end{equation*}
and
\begin{equation*}
A^{\varepsilon, X}(s,x):= u_o(X(s)) \; e^{\mathbf{W}(g_{s,x}^{\varepsilon, X})}\,.
\end{equation*}
Hence we have
\begin{equation*}
\begin{split}
V_{2,\varepsilon} &= \int_0^t \langle \nabla u_\varepsilon(s,x) - \nabla u(s,x) , g_{s,x}^{\varepsilon}\rangle \mathrm{d}s\\
&= \int_0^t \mathbb{E}^x\bigl[\bigl\langle A^X(s,x) g_{s,x}^X - A^{\varepsilon,X}(s,x) g_{s,x}^{\varepsilon,X}\; ,\; g_{s,x}^{\varepsilon}\bigr\rangle \bigr] \mathrm{d}s\\
&= \int_0^t \mathbb{E}^x\bigl[\langle (A^X - A^{\varepsilon, X}) g^{\varepsilon,X} , g^{\varepsilon}\rangle + \langle A^X (g^X- g^{\varepsilon,X}) , g^{\varepsilon}\rangle \bigr] \mathrm{d}s\\
&= \int_0^t \mathbb{E}^x[(A^X - A^{\varepsilon, X}) \langle g^{\varepsilon,X} , g^{\varepsilon}\rangle] + \int_0^t \mathbb{E}^x[A^X \langle g^X- g^{\varepsilon,X} , g^{\varepsilon}\rangle] \mathrm{d}s \,.
\end{split}
\end{equation*}
Let
\begin{equation*}
\mathrm{P}_{1,\varepsilon}:=\int_0^t \mathbb{E}^x[(A^X - A^{\varepsilon, X}) \langle g^{\varepsilon,X} , g^{\varepsilon}\rangle]\mathrm{d}s
\end{equation*}
and
\begin{equation*}
\mathrm{P}_{2,\varepsilon}:=\int_0^t \mathbb{E}^x[A^X \langle g^X- g^{\varepsilon,X} , g^{\varepsilon}\rangle] \mathrm{d}s\,.
\end{equation*}
So we will show in two steps that each of these terms converge to zero in $\mathcal{L}^2$.\\
\textbf{Step I: Convergence of $\mathrm{P}_{1,\varepsilon}$}.
For the first term, using H\"older inequality for $\frac{1}{p}+\frac{1}{q}=1$ we have
\begin{equation*}
\mathbb{E}^x|(A^X - A^{\varepsilon, X}) \langle g^{\varepsilon,X} , g^{\varepsilon}\rangle|
\leq \bigl(\mathbb{E}^x|A^X - A^{\varepsilon, X}|^q\bigr)^{1/q} \bigl(\mathbb{E}^x |\langle g^{\varepsilon,X} , g^{\varepsilon}\rangle|^p\bigr)^{1/p}\,.
\end{equation*}
In fact equation \eqref{bound on suprimum} also proves that for any $p\geq 1$
\begin{equation*}
\sup_{s\in[0,t]} \mathbb{E} \, \mathbb{E}^x |A^X(s,x) - A^{\varepsilon,X}(s,x)|^p \longrightarrow 0 \qquad \text{as} \quad \varepsilon \downarrow 0\,.
\end{equation*}

So if we can show that $\mathbb{E}^x |\langle g^{\varepsilon,X} , g^{\varepsilon}\rangle|^p$ is bounded by some constant which depends only on $H$ and $t$ we are done because then
\begin{equation*}
\begin{split}
&\mathbb{E} \Bigl(\int_0^t \mathbb{E}^x\bigl(\,(A^X - A^{\varepsilon, X}) \langle g^{\varepsilon,X} , g^{\varepsilon}\rangle\,\bigr)\, \mathrm{d}s\Bigr)^2\\
& \qquad \leq \mathbb{E} \Bigl(\int_0^t \bigl(\mathbb{E}^x|A^X - A^{\varepsilon, X}|^q\bigr)^{1/q} \bigl(\mathbb{E}^x |\langle g^{\varepsilon,X} , g^{\varepsilon}\rangle|^p\bigr)^{1/p}\, \mathrm{d}s\Bigr)^2\\
& \qquad \curlyeqprec \int_0^t \mathbb{E}\,\bigl(\mathbb{E}^x|A^X - A^{\varepsilon, X}|^q\bigr)^{2/q} \, \mathrm{d}s\,,
\end{split}
\end{equation*}
where $\curlyeqprec$ means {\it less than up to a constant}. So either $q>2$, where we get $\int_0^t (\mathbb{E}\,\mathbb{E}^x|A^X - A^{\varepsilon, X}|^q)^{2/q}\, \mathrm{d}s$ as an upper bound or $q\leq 2$, where we get the upper bound $\int_0^t \mathbb{E}\,\mathbb{E}^x|A^X - A^{\varepsilon, X}|^2 \, \mathrm{d}s$.

Let $\{t_i\}_{i=1}^n$ be the jump times of the path $X(\cdot)$ up to time $s$, $t_0:=0$ and $t_n:=s$. Let then $J$ be the set of indices $j$ for which $X(\cdot)$ stays at site $x$ in the time interval $[t_j,t_{j+1}]$. Now applying the definitions \eqref{gEpsilon}-\eqref{gEpsilonX} we get
\begin{equation*}
\begin{split}
\langle g^{\varepsilon,X} , g^{\varepsilon}\rangle &= \langle \sum_{i\in J} \int_{s-t_{i+1}}^{s-t_i} \frac{1}{2\varepsilon}\mathbf{1}_{[\theta-\varepsilon, \theta+\varepsilon]} \mathrm{d}\theta\; , \; \frac{1}{2\varepsilon}\mathbf{1}_{[s-\varepsilon, s+\varepsilon]}\rangle \\
&= \frac{1}{4 \varepsilon^2} \sum_{i\in J} \int_{s-t_{i+1}}^{s-t_i}\langle \mathbf{1}_{[\theta-\varepsilon, \theta+\varepsilon]} \; , \; \mathbf{1}_{[s-\varepsilon, s+\varepsilon]}\rangle \mathrm{d}\theta \\
&= \frac{1}{4 \varepsilon^2} \sum_{i\in J} \int_{s-t_{i+1}}^{s-t_i}  \mathbb{E}[(W_{\theta +\varepsilon}-W_{\theta -\varepsilon})(W_{s +\varepsilon} - W_{s -\varepsilon})]\mathrm{d}\theta \\
&= \frac{1}{8 \varepsilon^2} \sum_{i\in J} \int_{t_{i}}^{t_{i+1}} \bigl[(\gamma +2\varepsilon)^{2H} + |\gamma -2\varepsilon|^{2H} -2\gamma^{2H}\bigr]\mathrm{d}\gamma\,,
\end{split}
\end{equation*}
where $\{W_t\}_t$ is a fractional Brownian motion of the same Hurst parameter $H$. We split this expression into two terms
\begin{equation}\label{Gamma_1}
\Gamma_1:=\frac{1}{8 \varepsilon^2} \int_0^{t_1} \bigl[(\gamma +2\varepsilon)^{2H} + |\gamma -2\varepsilon|^{2H} -2\gamma^{2H}\bigr]\mathrm{d}\gamma
\end{equation}
and
\begin{equation*}
\Gamma_2:=\frac{1}{8 \varepsilon^2} \sum_{i\in J, i\geq 2} \int_{t_{i}}^{t_{i+1}}\bigl[(\gamma +2\varepsilon)^{2H} + |\gamma -2\varepsilon|^{2H} -2\gamma^{2H}\bigr]\mathrm{d}\gamma\,.
\end{equation*}

For the first term, using the same reasoning as in \eqref{first piece 1} and \eqref{first piece 2}, we have
\begin{equation}\label{integral form of integral of f superscript epsilon}
\Gamma_1 =\frac{1}{8}\int_{-1}^1\int_{-1}^1 f''(t_1+\xi \varepsilon+\eta \varepsilon)\,\mathrm{d}\xi\,\mathrm{d}\eta\,,
\end{equation}
where $f(s):=\int_0^s |r|^{2 H}\mathrm{d}r$ and hence $f''(r)=2H \,\sgn(r) |r|^{2H-1}$.\\
Letting $\Delta:=\xi \varepsilon+\eta \varepsilon$ and noting that $t_1$ is exponentially distributed, we have
\begin{equation*}
\mathbb{E}^x |f''(t_1+\Delta)|^p \leq 2H \int_0^s |t_1+\Delta|^{(2H-1)p} \mathrm{d}t_1\,.
\end{equation*}
As we can restrict ourselves to $\varepsilon \leq 1$ and hence $|\Delta|\leq 1$ and as  $0<s<t$, we have
\begin{equation*}
\int_0^s |t_1+\Delta|^{(2H-1)p} \mathrm{d}t_1 \leq \int_{-1}^{t+1} |t_1|^{(2H-1)p} \mathrm{d}t_1\,.
\end{equation*}
So if we choose $p>1$ such that ${(2H-1)p}>-1$, we get a finite bound on $\mathbb{E}^x |f''(t_1+\Delta)|^p$ and hence a bound on $\mathbb{E}^x|\Gamma_1|^p$ that only depends on $t$ and $H$.

Now for the second term, $\Gamma_2$, let
\begin{equation}\label{f superscript epsilon}
f^\varepsilon(\gamma):=\frac{1}{4 \varepsilon^2}\bigl[(\gamma +2\varepsilon)^{2H} + |\gamma -2\varepsilon|^{2H} -2\gamma^{2H}\bigr]\,.
\end{equation}
We have $\displaystyle|f^\varepsilon(\gamma)|\leq 18 \gamma^{2H-2}$ because either $\gamma \leq 4\varepsilon$ which implies that $|\gamma -2\varepsilon|^{2H}\leq (2\varepsilon)^{2H}$ and $(\gamma +2\varepsilon)^{2H}\leq (6\varepsilon)^{2H}$ and hence $|f^\varepsilon|(\gamma)\leq 18 \gamma^{2H-2}$ or $\gamma > 4\varepsilon$ in which case we may write $f^\varepsilon(\gamma)$ as the following
\begin{equation}\label{integral form of f supersript epsilon}
f^\varepsilon(\gamma)=\frac{1}{4}\int_{-1}^1 \int_{-1}^1 2H(2H-1)(\gamma + \xi \varepsilon+\eta \varepsilon)^{2H-2} \,\mathrm{d}\xi\,\mathrm{d}\eta\,.
\end{equation}
Letting again $\Delta:=\xi \varepsilon+\eta \varepsilon$, we have $|\Delta|\leq 2\varepsilon$ and so
\begin{equation*}
(\gamma+\Delta)^{2H-2} \leq \gamma^{2H-2} (1+\Delta/\gamma)^{2H-2} \leq 2^{2-2H} \gamma^{2H-2}\,,
\end{equation*}
which gives $\displaystyle|f^\varepsilon(\gamma)|\leq 8 \gamma^{2H-2}$.\\
So we have
\begin{equation*}
\Gamma_2 \curlyeqprec \int_{t_1}^s |f^\varepsilon(\gamma)|\mathrm{d}\gamma
\curlyeqprec \int_{t_1}^s \gamma^{2H-2}\,\mathrm{d}\gamma\,.
\end{equation*}
So $\Gamma_2$ is bounded (up to a constant) by either $ t_1^{2H-1}$ for $H<\frac{1}{2}$, or $s^{2H-1}$ for $H>\frac{1}{2}$. The case $H=\frac{1}{2}$ can also be treated easily using the inequality $ln(x)\curlyeqprec x^\alpha$ for any $\alpha$ positive. So as $(2H-1)p>-1$, $\mathbb{E}^x|\Gamma_2|^p$ can be bounded by a constant only dependant on $t$ and $H$. So this competes the proof showing that $\mathbb{E}^x |\langle g^{\varepsilon,X} , g^{\varepsilon}\rangle|^p \leq C$, for some $p>1$ and $C$ a constant only dependant on $t$ and $H$.\\
\textbf{Step II: Convergence of $\mathrm{P}_{2,\varepsilon}$}.
For establishing the convergence of $\mathrm{P}_{2,\varepsilon}$ we will use the dominated convergence theorem.\\
In `step I' we showed that
\begin{equation*}
\langle g^{\varepsilon,X} , g^{\varepsilon}\rangle = \frac{1}{2} \sum_{i\in J} \int_{t_{i}}^{t_{i+1}} f^\varepsilon(r) \, \mathrm{d}r\,,
\end{equation*}
where $f^\varepsilon$ is defined in \eqref{f superscript epsilon}.\\

Now let $\{t_i\}_{i=0}^{n+1}$ and $J$ be as in `step I', i.e. $\{t_i\}_{i=1}^n$ be the jump times of the path $X(\cdot)$ up to time $s$, $t_0:=0$ and $t_n:=s$ and $J$ the set of indices $j$ for which $X(\cdot)$ stays at site $x$ in the time interval $[t_j,t_{j+1}]$. So we have
\begin{equation}\label{integral of h superscrit epsilon}
\begin{split}
\langle g^X , g^{\varepsilon}\rangle &= \bigl \langle\mathbf{1}_{[0,s]}(r)\, \delta_{X(s-r)}(z)
\; , \; \frac{1}{2\varepsilon} \mathbf{1}_{[s-\varepsilon,s+\varepsilon]}(r)\,\delta_x(z) \bigr \rangle\\
&= \sum_{i\in J} \bigl \langle\mathbf{1}_{[{s-t_{i+1}\,,\,{s-t_i}}]}
\; , \; \frac{1}{2\varepsilon} \mathbf{1}_{[s-\varepsilon\,,\,s+\varepsilon]} \bigr \rangle\\
&= \sum_{i\in J} \frac{1}{4\varepsilon}\bigl[|t_{i+1}+\varepsilon|^{2H}-|t_i+\varepsilon|^{2H}
+|t_{i}-\varepsilon|^{2H}-|t_{i+1}-\varepsilon|^{2H}
\big]\\
&= \frac{1}{4\varepsilon}\bigl(|t_{1}+\varepsilon|^{2H}-|t_1-\varepsilon|^{2H}\big)+\frac{1}{2} \sum_{i\in J, i>1} \int_{t_{i}}^{t_{i+1}} h^\varepsilon(r) \, \mathrm{d}r\,,
\end{split}
\end{equation}
where
\begin{equation*}
h^\varepsilon(r):=\frac{2H}{2\varepsilon}\bigl[|r+\varepsilon|^{2H-1}
-\sgn(r-\varepsilon)|r-\varepsilon|^{2H-1}\bigr]\,.
\end{equation*}
We will show that $\langle g^X , g^{\varepsilon}\rangle - \langle g^{\varepsilon,X} , g^{\varepsilon}\rangle$ converges to zero. For doing so we shall show that $[\frac{1}{4\varepsilon}\bigl(|t_{1}+\varepsilon|^{2H}-|t_1-\varepsilon|^{2H}\big)-\frac{1}{2} \int_{0}^{t_{1}} f^\varepsilon(r) \, \mathrm{d}r]$ converges to zero and that every $\int_{t_{i}}^{t_{i+1}} (h^\varepsilon-f^\varepsilon)(r) \, \mathrm{d}r$ also converges to zero.\\

By equations \eqref{Gamma_1} and \eqref{integral form of integral of f superscript epsilon}, we have
\begin{equation*}
\int_0^{t_1} f^\varepsilon(r)\, \mathrm{d}r =\frac{1}{4}\int_{-1}^1\int_{-1}^1 2H \sgn(r+\xi \varepsilon+\eta \varepsilon)|r+\xi \varepsilon+\eta \varepsilon|^{2H-1}\,\mathrm{d}\xi\,\mathrm{d}\eta\,.
\end{equation*}
So for a fixed positive $t_1$ this converges to $2H t_1^{2H-1}$. On the other hand $\frac{1}{4\varepsilon}\bigl(|t_{1}+\varepsilon|^{2H}-|t_1-\varepsilon|^{2H}\big)$ also converges to $\frac{1}{2}2H t_1^{2H-1}$.

For $\int_{t_{i}}^{t_{i+1}} (h^\varepsilon-f^\varepsilon)(r) \, \mathrm{d}r$, we will show that $h^\varepsilon-f^\varepsilon$ converges to zero and then apply the dominated convergence to the integral.\\
Using \eqref{integral form of f supersript epsilon} it can be easily shown that \begin{equation*}
\lim_{\varepsilon \downarrow 0} f^\varepsilon(r) = 2H(2H-1)r^{2H-2}\,.
\end{equation*}
By simply recognizing the definition of derivative we have
\begin{equation*}
\lim_{\varepsilon \downarrow 0} h^\varepsilon(r) = 2H(2H-1)r^{2H-2}\,.
\end{equation*}

So it remains to find an integrable $\varepsilon$-independent upper bound. As shown in the paragraph following \eqref{f superscript epsilon}, $f^\varepsilon(r)$  is bounded by $18 \gamma^{2H-2}$ and for $h^\varepsilon(r)$, restricting $\varepsilon$ to be less than $t_{i_1}/2$, where $i_1$ is the first index in $J$ after $1$, we have for all $r\geq t_{i_1}$
\begin{equation}\label{integral form of h superscript epsilon}
h^\varepsilon(r)= \frac{1}{2}2H(2H-1) \int_{-1}^1 |r+u\varepsilon|^{2H-2}\, \mathrm{d}u\,.
\end{equation}
But then as $|r+u\varepsilon|^{2H-2}\leq (\frac{r}{2})^{2H-2}$ it gives $8 r^{2H-2}$ as an upper bound on $h^\varepsilon$. This completes the proof for convergence to zero of $\langle g^X , g^{\varepsilon}\rangle - \langle g^{\varepsilon,X} , g^{\varepsilon}\rangle$.

Now, for applying the dominated convergence theorem to $\mathrm{P}_{2,\varepsilon}$ we only need to find an $\varepsilon$-independent upper bound $G$ on $\langle g^X , g^{\varepsilon}\rangle - \langle g^{\varepsilon,X} , g^{\varepsilon}\rangle$ having the property that $\mathbb{E}\bigl(\int_0^t \mathbb{E}^x(G)\bigr)^2 < \infty$. For $\langle g^{\varepsilon,X}\rangle - \langle g^{\varepsilon,X} , g^{\varepsilon}\rangle$ such an upper bound has been established in step I above. It remains to find an upper bound on $\langle g^X , g^{\varepsilon}\rangle$.

For $2H-1\geq 0$ the situation is quite trivial because using equation \eqref{integral of h superscrit epsilon} we easily get
\begin{equation*}
\langle g^X , g^{\varepsilon}\rangle=\frac{1}{2} \sum_{i \in J} \int_{t_{i}}^{t_{i+1}} h^\varepsilon(r) \, \mathrm{d}r\,.
\end{equation*}
When $2H-1\geq 0$, equation \eqref{integral form of h superscript epsilon} remains valid for any value of $\varepsilon$ and $r$. As for any $\varepsilon\leq 1$ we have
\begin{equation*}
\int_{-1}^1 |r+u\varepsilon|^{2H-2}\, \mathrm{d}u \leq \int_{-1}^{t+1} |u|^{2H-2}\, \mathrm{d}u\,,
\end{equation*}
hence we get an upper bound dependant only on t and H.

So we consider now the case of $2H-1<0$. For $2H<1$ and any $r>0$ we have
\begin{equation*}
\rho(r):=\frac{1}{4\varepsilon}\bigl(|r+\varepsilon|^{2H}-|r-\varepsilon|^{2H}\big) \leq 2 r^{2H-1}.
\end{equation*}
This is true because either $r\leq 2\varepsilon$ in which case
\begin{equation*}
\begin{split}
\rho(r)&\leq \frac{1}{4\varepsilon}\bigl((3\varepsilon)^{2H}-\varepsilon^{2H}\big)\\
&\leq \varepsilon^{2H-1} \leq 2r^{2H-1}\,,
\end{split}
\end{equation*}
or $r> 2\varepsilon$, where we have
\begin{equation*}
\begin{split}
\rho(r)& \leq \frac{1}{4} \int_{-1}^{1}2H\,(r+\varepsilon u)^{2H-1}\, \mathrm{d}r\\
& \leq \frac{1}{4} \int_{-1}^{1}(\frac{r}{2})^{2H-1}\, \mathrm{d}r \leq r^{2H-1}\,.
\end{split}
\end{equation*}
So by \eqref{integral of h superscrit epsilon} we have
\begin{equation*}
|\langle g^X , g^{\varepsilon}\rangle|\leq 2\sum_{i\in J}(t_i^{2H-1}+t_{i+1}^{2H-1})\leq 2N t_1^{2H-1}\,,
\end{equation*}
where $N$ is the number of jumps in $[0,t]$.\\
Applying the H\"older inequality with $\frac{1}{p}+\frac{1}{q}+\frac{1}{r}=1$ we have
\begin{equation*}
\mathbb{E}^x|A^X \langle g^X , g^{\varepsilon}\rangle| \curlyeqprec
(\mathbb{E}^x\,|A^X|^q)^{1/q} (\mathbb{E}^x\,N\,^r)^{1/r} (\mathbb{E}^x\,t_1^{(2H-1)p})^{1/p}.
\end{equation*}
So we just need to pick a $p>1$ with $(2H-1)p+1>0$, in which case the exponential distribution of $t_1$ implies
\begin{equation*}
\mathbb{E}^x\,t_1^{(2H-1)p} \leq \int_0^s \,t_1^{(2H-1)p} \, \mathrm{d}t_1 = s^{(2H-1)p+1}\leq t^{(2H-1)p+1}\,.
\end{equation*}
In fact the proof of lemma \ref{uniform boundedness} also shows that for any $q\geq 1$, $\mathbb{E} \,\mathbb{E}^x\,|A^X|^{q}$ is uniformly bounded in $0\leq s \leq t$. As $N$ has a Poisson distribution $\mathbb{E}^x\,N^r$ is also finite.


\def\polhk#1{\setbox0=\hbox{#1}{\ooalign{\hidewidth
  \lower1.5ex\hbox{`}\hidewidth\crcr\unhbox0}}} \def\cprime{$'$}


\begin{thebibliography}{10}

\bibitem{Carmona}
Ren{\'e}~A. Carmona and S.~A. Molchanov.
\newblock Parabolic {A}nderson problem and intermittency.
\newblock {\em Mem. Amer. Math. Soc.}, 108(518):viii+125, 1994.

\bibitem{GaertnerKoenig}
J{\"u}rgen G{\"a}rtner and Wolfgang K{\"o}nig.
\newblock The parabolic {A}nderson model.
\newblock In {\em Interacting stochastic systems}, pages 153--179. Springer,
  Berlin, 2005.

\bibitem{Nualart2}
Yaozhong Hu, Fei Lu, and David Nualart.
\newblock Feynman-{K}ac formula for the heat equation driven by fractional
  noise with {H}urst parameter {$H<1/2$}.
\newblock {\em Ann. Probab.}, 40(3):1041--1068, 2012.

\bibitem{Nualart1}
Yaozhong Hu, David Nualart, and Jian Song.
\newblock Feynman-{K}ac formula for heat equation driven by fractional white
  noise.
\newblock {\em Ann. Probab.}, 39(1):291--326, 2011.

\bibitem{Janson}
Svante Janson.
\newblock {\em Gaussian {H}ilbert spaces}.
\newblock Cambridge University Press, Cambridge, 1997.

\bibitem{Kolmogorov}
A.~N. Kolmogoroff.
\newblock Wienersche {S}piralen und einige andere interessante {K}urven im
  {H}ilbertschen {R}aum.
\newblock {\em C. R. (Doklady) Acad. Sci. URSS (N.S.)}, 26:115--118, 1940.

\bibitem{LedouxTalagrand}
Michel Ledoux and Michel Talagrand.
\newblock {\em Probability in {B}anach spaces}.
\newblock Springer-Verlag, Berlin, 1991.

\bibitem{Mandelbrot}
Benoit~B. Mandelbrot and John~W. Van~Ness.
\newblock Fractional {B}rownian motions, fractional noises and applications.
\newblock {\em SIAM Rev.}, 10:422--437, 1968.

\bibitem{Nualart3}
David Nualart.
\newblock {\em The {M}alliavin calculus and related topics}.
\newblock Probability and its Applications (New York). Springer-Verlag, Berlin,
  second edition, 2006.

\bibitem{Protter}
Philip~E. Protter.
\newblock {\em Stochastic integration and differential equations}.
\newblock Springer-Verlag, Berlin, 2005.
\newblock Second edition. Version 2.1.

\bibitem{Sanz}
Marta Sanz-Sol{\'e}.
\newblock {\em Malliavin calculus with applications to stochastic partial
  differential equations}.
\newblock EPFL Press, Lausanne, 2005.

\bibitem{Stroock}
Daniel~W. Stroock.
\newblock {\em Markov processes from {K}. {I}t\^o's perspective}.
\newblock Princeton University Press, Princeton, NJ, 2003.

\bibitem{Twardowska}
Krystyna Twardowska.
\newblock Wong-{Z}akai approximations for stochastic differential equations.
\newblock {\em Acta {A}pplicandae {M}athematica}, 43(3):317--359, 1996.

\bibitem{Viens}
Frederi~G. Viens and Tao Zhang.
\newblock Almost sure exponential behavior of a directed polymer in a
  fractional {B}rownian environment.
\newblock {\em J. Funct. Anal.}, 255(10):2810--2860, 2008.

\end{thebibliography}
\end{document}